\documentclass[12pt,reqno]{amsart}
\usepackage{amsmath, amsthm, amssymb}
\usepackage[T1]{fontenc}
\usepackage[OT2, T1]{fontenc}
\usepackage[russian,english]{babel}
\usepackage{geometry}
\usepackage[colorinlistoftodos]{todonotes}
\usepackage{cancel}
\usepackage{soul}
\usepackage{comment}
\usepackage{color}
\usepackage{amscd}
\usepackage{tabularx}
\usepackage{url}
\usepackage{eurosym}
\usepackage{braket}
\usepackage{hyperref}
\usepackage{amsrefs}
\usepackage{setspace}
\usepackage{verbatim}
\usepackage{mathrsfs}
\usepackage{dsfont}
\usepackage{amsbsy}
\usepackage{amstext}
\usepackage{amsthm}
\usepackage{amssymb}
\usepackage{xcolor}
\usepackage{cancel}
\usepackage{graphicx}
\usepackage[export]{adjustbox}

\usepackage{float}
\PassOptionsToPackage{normalem}{ulem}
\usepackage[toc,page]{appendix}
\usepackage{graphicx,tikz}
\usepackage{enumerate}
\usepackage{adjustbox,graphicx,subcaption}
\makeatletter
\usepackage{eurosym}
\usepackage{mathrsfs}
\usepackage{dsfont}
\usepackage{xcolor}
\usepackage{cancel}

\newtheorem{theorem}{Theorem}[section]
\newtheorem{lemma}[theorem]{Lemma}

\newtheorem{conjecture}[theorem]{Conjecture}
\newtheorem{question}[theorem]{Question}

\theoremstyle{definition}
\newtheorem{definition}[theorem]{Definition}
\newtheorem{example}[theorem]{Example}

\newtheorem{rem}[theorem]{Remark}

\numberwithin{equation}{section}

\newcommand{\mmod}[1]{\,\,\textnormal{mod}\,\,#1}

\def \dx {x}

\def \bB {\mathbb B}

\def \bD {\mathbb D}
\def \bE {\mathbb E}
\def \bF {\mathbf F}

\def \bG {\mathbb G}
\def \bH {\mathbb H}
\def \I {\mathbb I}
\def \bK {\mathbb K}
\def \bN {\mathbb N}
\def \bP {\mathbb P}

\def \bR {\mathbb R}

\def \bZ {\mathbb Z}

\def \bT {\mathbb T}
\def \bM {\mathbb M}

\def \ba {\mathbf a}
\def \bb {\mathbf b}

\def \bf {\mathbf f}	
\def \bF {\mathbf F}
\def \bg {\mathbf g}
\def \bh {\mathbf h}

\def \bj {\mathbf j}

\def \bm {\mathbf m}

\def \bp {\mathbf p}
\def \bq {\mathbf q}

\def \bs {\mathbf s}
\def \bt {\mathbf t}

\def \bv {\mathbf v}
\def \bx {\mathbf x}

\def \by {\mathbf y}
\def \bz {\mathbf z}

\def\bfa{{\mathbf a}}
\def\ba{{\mathbf a}}
\def\bfb{{\mathbf b}}

\def\bff{{\mathbf f}}
\def\bf{{\mathbf f}}
\def\bfg{{\mathbf g}}

\def\bfj{{\mathbf j}}

\def\bfv{{\mathbf v}}

\def\bfx{{\dx}}
\def\bx{{\mathbf x}}
\def\bfy{{\mathbf y}}

\def\bfF{{\mathbf F}}

\def \rd {\mathrm d}

\def\cA{{\mathcal A}}

\def\cG{{\mathcal G}}
\def\cH{{\mathcal H}}

\def \fB {\mathfrak B}
\def \fC {\mathfrak C}

\def \fS {\mathfrak S}

\def \cA {\mathcal A}
\def \cB {\mathcal B}
\def \cC {\mathcal C}
\def \cD {\mathcal D}
\def \cE {\mathcal E}

\def \cG {\mathcal G}
\def \cH {\mathcal H}
\def \cI {\mathcal I}
\def \cJ {\mathcal J}

\def \cL {\mathcal L}
\def \cM {\mathcal M}

\def \cP {\mathcal P}

\def \cR {\mathcal R}
\def \cS {\mathcal S}
\def \cT {\mathcal T}

\def \cW {\mathcal W}

\def \cY {\mathcal Y}

\def \bzero {\mathbf 0}

\def \balp {{\boldsymbol{\alp}}}
\def \bbet {{\boldsymbol{\beta}}}

\def \bdel {{\boldsymbol{\del}}}
\def \bxi {{\boldsymbol{\xi}}}

\def\N{{\mathbb N}}
\def\R{{\mathbb R}}
\def\Z{{\mathbb Z}}\def\Q{{\mathbb Q}}

\def\bump{{\textcolor{black} {w}}}
\def \supp {{\mathrm{supp}}}

\def \ds1 {\mathds{1}}

\def\alp{{\alpha}} 
\def\bet{{\beta}}  
\def\gam{{\gamma}}

\def\del{{\delta}} \def\Del{{\Delta}}

\def\kap{{\kappa}}
\def\lam{{\lambda}}

\def\sig{{\sigma}}

\def\ome{{\omega}} 
\def\d{{\partial}}
\def\eps{\varepsilon}

\def\le {\leqslant}
\def\leq {\leqslant}
\def\ge {\geqslant}
\def\geq {\geqslant}

\def\d{{\,{\rm d}}}

\def\d{{\mathrm{d}}}
\def \max {{\mathrm{max}}}

\def\wgood{{\Omega_{\Delta, K, \mathrm{g} }}}

\def\wgoodpi{h_{N,\nu}}
\def\wbad{{\Omega_{\Delta, K, \mathrm{sub} }}}

\DeclareMathOperator{\dist}{dist}

\newcounter{@ToDo}
\newcommand{\mtodo@helper}[1]{%
	({\color{blue}TODO~\arabic{@ToDo}: {#1\@addpunct{.}}})%
}

\newcommand{\mtodo}[1]{\stepcounter{@ToDo}%
	\relax\ifmmode\text{\mtodo@helper{#1}}%
	\else\mtodo@helper{#1}\fi%
}

\newcounter{@cdo}
\newcommand{\cdo@helper}[1]{%
	({\color{red}CITE~\arabic{@cdo}: {#1\@addpunct{.}}})%
}
\newcommand{\cdo}[1]{\stepcounter{@cdo}%
	\relax\ifmmode\text{\cdo@helper{#1}}%
	\else\cdo@helper{#1}\fi%
}

\def \dimh {{\mathrm{\dim_H}}}

\def \det {\mathrm{det}}

\def \supp {{\mathrm{supp}}}

\def \d {{\mathrm{d}}}
\def \dist {{\mathrm{dist}}}

\author{Sam Chow}
\author{Rajula Srivastava}
\author{Niclas Technau}
\author{Han Yu}

\begin{document}


\global\long\def\bA{\mathbf{A}}%

\global\long\def\bB{\mathbf{B}}%


\global\long\def\bD{\mathbf{D}}%

\global\long\def\bE{\mathbf{E}}%

\global\long\def\bF{\mathbf{F}}%

\global\long\def\bG{\mathbf{G}}%

\global\long\def\bH{\mathbf{H}}%

\global\long\def\bI{\mathbf{I}}%

\global\long\def\bJ{\mathbf{J}}%

\global\long\def\bK{\mathbf{K}}%

\global\long\def\bL{\mathbf{L}}%

\global\long\def\bM{\mathbf{M}}%


\global\long\def\bO{\mathbf{O}}%






\global\long\def\bU{\mathbf{U}}%

\global\long\def\bV{\mathbf{V}}%

\global\long\def\bW{\mathbf{W}}%

\global\long\def\bX{\mathbf{X}}%

\global\long\def\bY{\mathbf{Y}}%


\subjclass[2020]{11J83; 11K55; 11J25; 42B20}
\keywords{Rational points near manifolds, multiplicative Diophantine approximation on manifolds, quantitative non-divergence, oscillatory integrals, Hausdorff dimension, spectrum of exponents}

\address{Sam Chow;
\newline
Mathematics Institute, Zeeman Building, University of Warwick, Coventry CV4 7AL, United Kingdom}
\email{sam.chow@warwick.ac.uk}

\address{Rajula Srivastava; 
\newline 
Department of Mathematics, 
University of Wisconsin 480 Lincoln Drive, 
Madison, WI, 53706, USA}
\email{
rsrivastava9@wisc.edu}

\address{Niclas Technau; 
\newline 
Mathematical Institute, University of Bonn, Endenicher Allee 60,
53115, Bonn, Germany;
\newline Department of Mathematics, 
University of Wisconsin 480 Lincoln Drive, 
Madison, WI, 53706, USA}
\email{technau@wisc.edu}

\address{Han Yu;
\newline
Mathematics Institute, Zeeman Building, University of Warwick, Coventry CV4 7AL, United Kingdom;
\newline 
College of Mathematics and Statistics, Center of Mathematics, Chongqing University, Chongqing, 401331, China}
\email{han.yu.2@cqu.edu.cn}

\title[Multiplicative Diophantine approximation on Manifolds]
{Rational Points in Hyperbolic Regions and Multiplicative Diophantine Approximation 
on 
Manifolds
}
\date{\today}


\begin{abstract}
We establish the convergence theory of 
multiplicative Diophantine approximation for all non-degenerate, smooth manifolds. We also settle said convergence theory for all affine subspaces satisfying 
a highly generic and essentially optimal Diophantine condition. 
This answers a question of Beresnevich and Velani
from 2005, while simultaneously
sharpening results of Kleinbock and Margulis
on the strong extremality of non-degenerate manifolds, and of Kleinbock on the strong extremality of affine subspaces. 
\end{abstract}
\maketitle

\section{Introduction and Main Results}

\subsection{Introduction}

\subsubsection{Simultaneous and multiplicative approximation}

We investigate rational approximations to points in $\bR^n$, where $n \in \bN$ is fixed. It follows from Dirichlet's approximation theorem that if $\bx = (x_1,\ldots, x_n) \in \bR^n$ then there exist infinitely many $q \in \bN$ such that
\[
\max \{ \| q x_1 \|, \ldots, \| q x_n \| \} < q^{-1/n},
\]
where $\| \cdot \|$ denotes the distance to $\bZ$. This describes the rate at which $x_1, \ldots, x_n$ can be simultaneously approximated by rationals of the same denominator. By Khintchine's theorem \cites{Khi1924, Khi1926}, if $\psi: \bN \to [0,1)$ is monotonic and $\sum_{q=1}^\infty \psi(q)^n = \infty$ then, for almost all $\bx \in \bR^n$, there exist infinitely many $q \in \bN$ such that
\[
\max \{ \| q x_1 \|, \ldots, \| q x_n \| \} < \psi(q).
\]
This is the cornerstone of the metric theory of diophantine approximation. The result is sharp, for if the series converges then the set of such $\bx$ has measure zero.

We see from the discussion above that if $\bx \in \bR^n$ then there exist infinitely many $q \in \bN$ such that
\[
\| q x_1 \| \cdots \| q x_n \| < 1/q.
\]
This describes the multiplicative rate at which $x_1, \ldots, x_n$ can be approximated by rationals of the same denominator. Famously, Littlewood's conjecture asserts that if $n \ge 2$ then $1/q$ can be replaced by $c/q$ here, for any $c > 0$. Gallagher's theorem \cite{Gal1962} asserts that if $\psi: \bN \to [0,1)$ is monotonic then almost no/every $\bx \in \bR^n$ admits infinitely many $q \in \bN$ such that
\[
\| q x_1 \| \cdots \| q x_n \| < \psi(q)
\]
if the series \begin{equation}
\label{eq: series of measures}
\sum_{q\geq 1}\psi(q)(\log q)^{n-1}
\end{equation} converges/diverges.

\subsubsection{Diophantine approximation on manifolds}

Let $\cM$ be an analytic submanifold of $\bR^n$. This comes with an induced Lebesgue measure, and one can ask for analogues of the aforementioned theorems of Khintchine and Gallagher. This is an old problem which appears to have had its genesis in a 1932 paper by Mahler \cite{Mah1932}, in which metric diophantine approximation on the Veronese curve $\{ (t, t^2, \ldots, t^n) \}$ is discussed.

There is often a difference in behaviour according to whether $\cM$ is flat or curved. To be precise, we say that $\cM$ is \emph{non-degenerate} if it is not contained in any proper affine subspace of $\bR^n$. The definition can be extended to sufficiently smooth manifolds that are not necessarily analytic, see \cite{KM1998} and Definition \ref{def nondeg} below.

The Khintchine divergence theory was resolved for $C^3$, non-degenerate planar curves in \cite{BDV2007}, and for all analytic, non-degenerate manifolds in \cite{Ber2012}. The assumption of analyticity was removed for curves in \cite{BVVZ2021}, and recently for all non-degenerate manifolds in \cite{BD}. 

The corresponding convergence theory was solved for $C^3$, non-degenerate planar curves in \cite{VV2006}. It was recently settled for all non-degenerate manifolds by Beresnevich and Yang \cite{BY2023}, making essential use of the quantitative non-divergence estimates developed in \cite{KM1998} and \cite{BKM2001}.

In the complementary setting of affine subspaces, the convergence theory requires a diophantine assumption on the parametrising matrix. See the recent article of Huang \cite{Hua2024}, which built on the work of Kleinbock \cite{Kle2003} and Huang--Liu \cite{HL2021}.

Naturally, diophantine approximation on manifolds is closely related to the density of rational points near manifolds. The latter also has a long and distinguished history, and continues to be vigorously investigated \cites{BZ2010, Cho2017, Gaf2014, Hua2015, Hua2019, Hua2020, SST, SY2022, sri2025counting, ST}.


\subsection{Main results}

The analogue of Gallagher's theorem has been much more elusive. Badziahin and Levesley \cite{BL2007} settled the convergence theory for non-degenerate planar curves. In \cite[Problem S2]{BV2007}, released in 2005, Beresnevich and Velani posed the problem of generalising this. In the flat setting there has been substantial progress, for fibres \cites{BHV2020, Cho2018, CT2019, CT2024, Yu2023}, planar lines \cites{CY2024, Hua2024}, and affine hyperplanes \cite{CY1}. For hypersurfaces with $n \ge 5$ and non-vanishing Gaussian curvature, the first and fourth named authors recently established the Gallagher property in full \cite{CY1}. 
 
\begin{definition}
We say
$\bx\in\R^{n}$ is 
$\psi$-multiplicatively approximable if 
\begin{equation}
\label{def mult app cond}
\prod_{i\leq n}\Vert q\dx_{i}\Vert<\psi(q)
\end{equation}
holds for infinitely many $q \in \bN$. 
Let $\cW^{\times}_n(\psi)$ be
the set of $\psi$-approximable vectors 
in $\R^{n}$.
\end{definition}

\begin{definition}
A Borel measure $\mu$ on $\bR^n$ is \emph{convergent Gallagher} if, for any non-increasing function $\psi:\bN \rightarrow [0,1)$,
\[
\sum_{q=1}^\infty \psi(q) 
(\log q)^{n-1} < \infty
\quad
\Longrightarrow 
\quad
\mu(\cW_n^\times(\psi)) = 0.
\]
\end{definition}

In this paper, we are interested in the set of $\psi$-approximable points on a smooth, flat or curved (non-degenerate) submanifold of $\mathbb{R}^n$. 
For an integer $1\leq d\leq n$, consider 
\begin{equation}\label{def U}
V\subseteq U\subset \mathbb{R}^d
\end{equation}
so that $V$ is closed set and $U$ a bounded, open set.
Our primary object of interest is the manifold
$$\cM =\{\mathbf{f}(\bx): \bx\in V
\},$$ 
embedded by a sufficiently smooth function
$
\mathbf{f} = (f_1, \ldots, f_n): U\to \mathbb{R}^n.
$
Note that we allow $f_1,\dots,f_n$ to be linear. In this `flat' case, we denote the manifold by $\cL$. Let us now describe the `curved' setting precisely.

\begin{definition}
\label{def nondeg}
Let $1\leq l\leq n-d$ be an integer. We say that $\bx$ 
is an $l$-non-degenerate
point of $\cM$ if the partial derivatives of 
$\mathbf{f}$ at $\bx$ of order in $[1,l]$ span $\mathbb{R}^n$. We say that $\mathbf{f}(\bx)$ is non-degenerate if it is $l$-non-degenerate for some $l\leq n-d$.
\end{definition}

\begin{rem}
The non-degeneracy condition is an infinitesimal version of $\cM$ not lying in any proper
affine hyperplane, i.e. the linear independence of $1, f_1, \ldots, f_n$ over $\mathbb{R}$. For example,
if the functions $f_1, \ldots, f_n$ are analytic --- in other words, the manifold $\cM$ is analytic --- then the linear independence of $1, f_1, \ldots, f_n$ over $\mathbb{R}$
in the domain $U$ is equivalent to all points of $\cM$ being
non-degenerate.
\end{rem}

We establish the following results.

\begin{theorem}
\label{thm: Gallagher conv}
Let $\cM\subseteq\R^{n}$ be a smooth, non-degenerate manifold of dimension $d\geq 1$. Then the natural Lebesgue measure on $\cM$ is convergent Gallagher.
\end{theorem}

\begin{rem}
We establish this indirectly. We use a fibring trick (Appendix~\ref{sec fibering}) to reduce the theorem to that concerning smooth, non-degenerate curves (Theorem \ref{thm: main curves}). More precisely, we will prove a fibring lemma (Lemma \ref{lem fiber}) for \textbf{sufficiently smooth,} non-degenerate manifolds which shows that any $d$-dimensional smooth, non-degenerate manifold $\cM$ can be decomposed into smooth, non-degenerate curves. Such a result has been well known and often used for \textbf{analytic} manifolds and curves, see \cites{Ber2015, sprindzhuk1980achievements}. 

\end{rem}

\begin{theorem}
\label{FlatConv}
Let $1\leq d < n$ be positive integers.
Let $\cL$ be a $d$-dimensional affine subspace of $\bR^n$ satisfying Hypothesis D.
Then the induced Lebesgue measure on $\cL$ is convergent Gallagher.
\end{theorem}

\begin{rem}
Hypothesis D is a diophantine condition on the defining parameters of $\cL$ that will be described in \S\ref{FlatSection}. We will see that it is almost sharp, and that it is generic.
\end{rem}

\subsection{Further remarks and open problems}
Theorems \ref{thm: Gallagher conv} and \ref{FlatConv} are already new, for example, for curves/lines in $\bR^3.$ While these results account for most of the manifolds, there are still some cases that remain unexplored.

\begin{conjecture}\label{conj: rel non-deg}
Let $\cL\subset\mathbb{R}^n$ be an affine subspace satisfying Hypothesis D. Let $\cM\subset\cL$ be a smooth manifold which is non-degenerate when viewed as a manifold in $\mathbb{R}^{\dim \cL}\simeq \cL.$ Then the induced Lebesgue measure on $\cM$ is convergent Gallagher.
\end{conjecture}
\begin{rem}
This is closely related to an `inheritance' property discussed in \cite{Kle2003}, where it was shown that analytic manifolds inherit the extremality from their supporting affine subspaces, i.e. the smallest containing affine subspace. In the case $\cL = \bR^n$, the conjecture specialises to Theorem \ref{thm: Gallagher conv}.
\end{rem}

Another curious problem concerns the monotonicity condition on the approximation function $\psi$. Notice that the convergent Gallagher property is satisfied by the ambient Lebesgue measure, i.e. on $\mathbb{R}^n$, without the monotonicity condition. It is unclear whether the monotonicity condition is necessary in the convergence theory of manifolds. 
\begin{question}
In Theorems \ref{thm: Gallagher conv} and \ref{FlatConv}, is it possible to drop the monotonicity condition on $\psi$?
\end{question}

\begin{rem}
For simultaneous approximation, \cite[Theorem 2.3]{Hua2024} provides a sufficient condition for the non-monotonic convergence property for affine subspaces. In our multiplicative context, it is likely that a non-monotonic version of flat Theorem \ref{FlatConv} can be supplied. However, the non-monotonic analogue of the curved Theorem \ref{thm: Gallagher conv} is largely open. There are some existing results along these lines, e.g. \cites{BVVZ2017, CY1, HuangPara}.
However, we do not know of any example of a non-degenerate manifold which fails the non-monotonic convergence theory, either simultaneous or multiplicative.
\end{rem}

Finally, we note that our method is robust enough to handle variations of the approximation problem, e.g. weighted approximation and, in some cases, inhomogeneous approximation. We choose not to pursue this level of generality in this paper.

\begin{rem}

While finishing this paper, we became aware of related work of Beresnevich, Datta, 
and Yang \cite{BDY2026} on weighted and multiplicative Diophantine approximation  for non-degenerate manifolds.
Their approach uses dynamical methods and the geometry of numbers, building upon  \cite{Ber2012,BY2023}. 
\end{rem}

\subsection{Organisation} 

We discuss some main ideas of the proofs in \S \ref{sec: outline}. After that, we dive into the full proofs.

\bigskip

For Theorem \ref{FlatConv}, we require the notion of Hypothesis D. The hypothesis, and its important properties, are provided in \S \ref{FlatSection}. The rest of \S \ref{FlatSection} contains a proof of Theorem \ref{FlatConv}.

\bigskip

For Theorem \ref{thm: Gallagher conv}, we need the convergent Gallagher property for smooth, non-degenerate curves in $\mathbb{R}^n$. It suffices to establish this for a curve $\cC$ parametrised over a bounded, open set in $\mathbb{R}$. More precisely, let $\I\subset \mathbb{R}$ be a closed interval contained in the open interval $U\subset \mathbb{R}$. Consider the curve 
\begin{equation}
\label{eq: curve}
\cC = \{\mathbf{f}(x): x \in \I\subset U\}, 
\end{equation}
where 
$
\mathbf{f} =(f_1, \ldots, f_n): U\to \mathbb{R}^n
$
is smooth and non-degenerate. In \S \ref{sec: proof of curves}, we prove the following theorem.
\begin{theorem}
\label{thm: main curves} 
If  the curve $\cC$ from 
\eqref{eq: curve} is non-degenerate, then the natural Lebesgue measure on $\cC$ is convergent Gallagher.
\end{theorem}
Theorem \ref{thm: Gallagher conv} then follows from Theorem \ref{thm: main curves} and the fibring trick in Appendix \ref{sec fibering}. The proof of Theorem \ref{thm: main curves} requires some technical arguments from harmonic analysis, which we provide in Appendices \ref{app proof lemma} and \ref{sec proof lemma pruning}.

\subsection{Notation}

For complex-valued functions $f$ and $g$, we write $f \ll g$ or \mbox{$f = O(g)$} if there exists $C$ such that $|f| \le C|g|$ pointwise, we write $f \sim g$ if $f/g \to 1$ in some specified limit, and we write $f = o(g)$ if $f/g \to 0$ in some specified limit. We will work in $n$-dimensional Euclidean space, and the implied constants will always be allowed to depend on $n$, as well as on any parameters which we describe as fixed or constant. Any further dependence will be clarified with a subscript or in words.

We denote $\bT = \bR/\bZ$.
For $x \in \bR$, we write $e(x) = e^{2 \pi i x}$. For $r > 0$ and $\by \in \bR^n$, we write $B_r(\by)$ for the ball of radius $r$ centred at $\by$, and put $B_r = B_r(\bzero)$. For $k \in \bN$, we write $\lam_k$ for $k$-dimensional Lebesgue measure. For $d \in \bN$, we write $\mu_d$ for the induced $d$-dimensional Lebesgue measure on a $d$-dimensional manifold embedded into a Euclidean space.
For $\tau > 0$, we write $\cH^\tau$ for $\tau$-dimensional Hausdorff  measure.

Bold face letters, such as $\bx$, always denote vectors. The components of $\bx$ are denoted by
$x_1,\ldots,x_n$. The variable $n$ is reserved for
the dimension of the ambient Euclidean space. 
We shall use the shorthand notation 
$$
\I:=\{\dx\in \mathbb{R}: |\dx|\leq 1\}.
$$
Given a scale $2^{t}$, we group all $Q:=Q_{t}:=2^{t}\leq q<2^{t+1}$
together (smoothly). 
For a vector $\ba\in \Z^n$, we write $\ba'$ 
for its last $n-1$ coordinates.

\section{Outline of the proofs}\label{sec: outline}

Our proofs blend many different ideas from previous works \cites{CY1, Hua2024, SST, sri2025counting}. To gain a coherent understanding, we first discuss the central idea through both the proofs. After that, more specific techniques diverge for the flat and curved cases. 

\subsection{A high-level overview}

Let $\cM\subset\mathbb{R}^n$ be a smooth manifold, and let $\bx\in\cM$. Let $Q>0$ be a large integer. We wish to study the distribution of $\|q\bx\|$ for $q\in [Q,2Q]$. Given the nature of the approximation problem --- Khintchine or Gallagher --- we find a target set $\cT_Q\subset [0,1]^n$ controlled by the approximation function at height $Q$, i.e. $\psi(Q)$, and estimate the target-hitting statistic 
\[
\{q\in [Q,2Q]:\|q\bx\|\in \cT_Q\}.
\]

For the Khintchine problem, we choose $\cT_Q$ to be a Euclidean ball of radius $\psi(Q)^{1/n}$. For the Gallagher problem, we choose $\cT_Q$ to be a hyperbolic region whose Lebesgue measure is $\psi(Q) \log^{n-1} (1/\psi(Q))$. We further decompose $\cT_Q$ into a union of rectangles of various dimensions. Existing methods, e.g. \cites{BY2023, SST, sri2025counting}, enable us to handle the case when $\cT_Q$ is a rectangle that is not too small, or a finite union of such rectangles.
In the case of counting rational points near curved manifolds, this was done 
for the first time by the second named author \cite{sri2025counting}, albeit under a rigid geometric condition on the manifold. To study the Gallagher problem, one needs to study the target-hitting statistics in a uniform manner for all possible rectangular $\cT_Q$. It is this uniformity that requires new inputs. In this paper, we illustrate how to handle it via Fourier analysis.

Considering only $\|q\bx\|_{q\in [Q,2Q]}$ for each point $\bx\in \cM$, hardly gives us any new information. 
To take advantage of the structure of $\cM$, we instead consider a suitable neighbourhood of the point $\bx$. Let $T_\bx$ be the affine subspace tangent to $\cM$ at $\bx$. For $\delta>0$, consider the $\delta$-ball $B_{\delta,\bx}$ around $\bx$ in $\cM$. Owing to the smoothness of $\cM$, the sets $T_{\bx}$ and $B_{\delta,\bx}$ are equivalent at scale $\delta^{1/2}$, i.e. their Hausdorff distance is $O(\delta^{1/2})$. We choose a suitable value of $\delta$ according to $Q$. Instead of looking at $\|q\bx\|,$ we look at $\|q B_{\delta,\bx}\|$, which is essentially $\|q T_{\bx}\|$ up to a suitable cutoff.

Since $T_\bx$ is linear, our first task is then the study of diophantine approximation for points in affine linear subspaces. When $\cM$ is flat, we shall see that under a mild condition (Hypothesis D), the sequence $\|qT_\bx\|_{q\in [Q,2Q]}$ effectively equidistributes in $[0,1]^n$, uniformly across a large collection of rectangular $\cT_Q$. From here, we obtain the Gallagher convergence theory for such affine subspaces. This is carried out in \S\ref{FlatSection}. 

\bigskip

In the curved setting, we can decompose $\cM$ into small and approximately flat pieces. It would be convenient if all such flat pieces were to satisfy Hypothesis D, but this is not true in general. To overcome this issue, we assume that $\cM$ is non-degenerate and invoke the quantitative non-divergence estimate for non-degenerate manifolds --- due to Bernik--Kleinbock--Margulis, see Theorem \ref{thm quant non div} --- which states, roughly speaking, that most points $\bx\in\cM$ satisfy a certain diophantine condition involving $\bx$ and $T_\bx$. We decompose $\cM$ according to the quantitative non-divergence rate; such a decomposition is referred to as sub-level set decomposition, and its simultaneous approximation version was introduced in \cite{SST}. Recall that we need to handle target-hitting statistics uniformly for many rectangular $\cT_Q$. The decomposition is in general different for different rectangular $\cT_Q.$ For a uniform treatment, one needs to carefully track this dependency. This technical issue will be taken care of in \S\ref{subsec smooth cutoff}. The uniformity is illustrated in \S\ref{sec final reduc}.

These operations furnish a family of decompositions --- one for each rectangular $\cT_Q$ --- of $\cM$ into small almost-linear pieces, and after dropping a small number of them, we can assume that all these almost-linear pieces satisfy a certain diophantine property. We exploit the diophantine property to obtain equidistribution for $\|q\cup_{\bx} T_{\bx}\|_{q\in [Q,2Q]}.$ This is done in \S\ref{subsec fourier expansion} via Fourier analysis in a similar manner to \cite{SST}. In doing so, we also bookkeep the dependence between the shape of $\cT_Q$ and the decomposition on $\cM$. From here, we are able to establish the convergence Gallagher property for non-degenerate manifolds.

\begin{rem}
We restrict attention to non-degenerate manifolds because we rely on the quantitative non-divergence estimate of Bernik, Kleinbock, and Margulis. We believe that the latter should hold for relatively non-degenerate manifolds in diophantine affine subspaces, potentially leading to inroads into Conjecture \ref{conj: rel non-deg}.
\end{rem}

\subsection{The flat case, i.e. Theorem \ref{FlatConv}}
We start by expressing $\cW_n^\times(\psi)$ as a lim sup set of hyperbolic regions indexed by $q \in \bN$. Collecting together $q \in (Q,2Q]$ for $Q$ a power of two, we index by $Q$ and eventually apply the first Borel--Cantelli lemma. This reduces our task to bounding the induced Lebesgue measure of 
\[
\bigcup_{Q \le q \le 2Q}
\{ \bt \in \fB:
(q \bt, q(A\bt + \bb)) \in \cT \},
\]
where $\fB$ is a ball in $\bR^d$ and $\cT = \{ \bx \in \bT^n: \| x_1 \| \cdots \| x_n \| \le \psi(Q) \}$. We can cover $\cT$ by axis-aligned rectangles $\cA$ that are not too small. These steps are executed in \S\ref{completing}.

Next, we observe that $q \bt$ vanishes on $\bT^d$ whenever $\bt$ is $q^{-1}$ times an integer vector. This enables us to decompose the flow $(q\bt, q(A\bt + \bb))$ into small regions $\wp$. Here $\wp$ begins at $(\bm, \bp)$, where $\bm \in \bZ^d$, and $\bp$ is some 
$O(Q)$-bounded $\bZ$-linear combination of $\bb$ and the columns of $A$. This reduces our task to counting how many small regions $\wp$ intersect $\cA$. The key observation is that $\wp$ can only intersect $\cA$ if $\bp$ lies in a certain projection of $\cA$. These steps are executed in \S\ref{TargetHitting}.

Finally, we need to bound how many times $\bp$ lies in this projected set. The latter is efficiently contained within an axis-aligned rectangle, and the requisite count matches the natural equidistribution heuristic. Using smooth weights and Fourier series, this expectation is realised by the contribution of the zero frequency. Our smooth weights are blessed with rapid Fourier decay, and bounding the contributions from the non-zero frequencies boils down to dual diophantine approximation for $\bb$ and the columns of $A$. Hypothesis D implies a gap principle which enables us to successfully bound these contributions. These steps are executed in \S\ref{equidistribution}.

\subsection{The curved case, i.e. Theorem \ref{thm: Gallagher conv}}

Using a fibring trick (Appendix \ref{sec fibering}) and the non-degeneracy of $\cM$, we first reduce matters to a model class of smooth, non-degenerate curves in \S\ref{subsec model class}. 
Next, as described  in \S\ref{subsec psi simp}, to estimate the number of
$\psi(t)$-good multiplicative approximations 
near  
$\cC$ for each integer $t$, we need to 
count lattice points in \textit{non-isotropic} dyadic rectangles around $\cC$ of dimensions $\bdel=(2^{-z_1}, \ldots, 2^{-z_n})$ and volume
$$2^{-z_1+\ldots+z_n}\asymp \psi(2^t)\asymp 2^{-t(1+o(1))}.$$

In \S \ref{sub sec counting prob}, we pass to this non-isotropic counting function  
and in \S\ref{subsec conv mod counting}, we reduce the proof of Theorem \ref{thm: main curves} to that of our main counting result, Theorem \ref{thm: prob count}. Our next aim is to show that, modulo a set of small measure, all of the relevant non-isotropic rectangles around $\cC$ contain the expected number of rational points. 
We pass to a smooth counting function in \S \ref{subsec smooth cutoff} and make some final technical reductions in \S\ref{sec final reduc}.

After two Poisson summations,
this counting function can be reinterpreted
as 
the Fourier transform of
the surface measure of 
$\cT\subseteq \R^{1+n}$
given by $(y,x)\mapsto y(1,\bf(x))$, where $\bf$ is as in \eqref{eq: curve}. 
With this in mind, we smoothly decompose the curve $\cC=\cA_{\bdel}\cup \cB_{\bdel}$
using non-standard bump functions adapted to small scales depending on $t$ and $\bdel$. This is carried out in \S \ref{sub-levelFA}.

The set $\cA_{\bdel}$ is the 
portion of $\cC$ which lifts to the part of $\cT$ where the Fourier transform of its surface measure
does not have rapid decay.
Building on the work of Schindler and two of the authors \cite{SST},
we reinterpret this as a sub-level set estimation problem in \S \ref{subsec sub-level}. To bound the measure of $\cA_{\bdel}$, we use the seminal work
of Bernik, Kleinbock and Margulis from \cite{BKM2001} on quantitative non-divergence in the space of lattices (see Theorem \ref{thm quant non div}). This is carried out in \S \ref{subsec quant nondiv}.

To establish good Fourier decay of the surface measure on the complementary part $\cB_\del$, 
we estimate a family of oscillatory integrals whose amplitudes are the aforementioned bump functions whose derivatives blow up as in \eqref{eq r rough est} and \eqref{eq amp blowup}. 
This is achieved in \S \ref{subsec fourier expansion}. The broad strategy here is similar to that in \cite{SST}. However, the non-isotropic/multiplicative nature of problem at hand requires more delicate estimates, leading to new constraints on the parameters in the quantitative non-divergence estimate and the scale of the bump functions, all of which are functions of $\bdel$ and $t$. In \S \ref{subsec proof counting}, we choose these parameters and combine our arguments to prove Theorem \ref{thm: prob count}.

\section{The flat setting}
\label{FlatSection}

In this section, we first discuss Hypothesis D --- including its sharpness and genericity --- and 
prove Theorem \ref{FlatConv}.

\subsection{Diophantine hypothesis}

Let $\cL$ be $d$-dimensional. We assume that the projection of $\cL$ onto any axis-aligned subspace of dimension $d$ is surjective. Then $\cL$ has the Monge parametrisation
\begin{equation}
\label{Monge}
\cL =
\left \{
\begin{pmatrix}
\bt \\
A\bt + \bb
\end{pmatrix}
: \bt \in \bR^d \right \},
\end{equation}
where 
$$
A = (\ba^{(1)}, 
\ldots, \ba^{(d)}) \in \bR^{(n-d)\times d}
\quad \mathrm{and}
\quad \bb \in \bR^{n-d}.
$$
Moreover, by our projection-surjectivity assumption, $\cL$ has such a parametrisation in each of the $n!$ many orderings of the standard basis on $\bR^n$. To elaborate, in \eqref{Monge} we treated the first $d$ coordinates as being special. However, we can deem any choice of $d$ coordinates to be special, resulting in different Monge forms. For example, in the case $(n,d) = (2,1)$, we can present a planar line as $y=ax+b$ or $x=a^{-1}y-a^{-1}b.$

Hypothesis D will be described in terms of Hypothesis $D(c, \sig_1, \sig_2)$, for certain parameters $c,\sig_1, \sig_2$. For $k \in \bN$ and $\bm \in \bZ^k$, we write
\[
|\bm| = \max \{ |m_j|: 1 \le j \le k \}, \qquad
H(\bm) = \prod_{j \le k} \max \{ 1, |m_j| \}.
\]

\begin{definition}
Let $\cL$ be as described at the beginning of this section. Let $c > 0$, and let $\sig_1,\sig_2 \in (0, \infty]$. 
We say $\cL$ satisfies 
Hypothesis $D(c, \sig_1,\sig_2)$ 
if, for any Monge parametrisation and any $T \ge 1$, the inequality
\begin{align}
\label{Cond: D}
\max \{ \| \bxi \cdot \ba^{(1)} \|,
\ldots, \| \bxi \cdot \ba^{(d)} \|, \| \bxi \cdot \bb \| \} \geq cT^{-\sigma_1} 
\end{align}
holds whenever
\begin{equation}
\label{whenever}
\bzero \ne \bxi \in \bZ^{n-d}, \qquad
|\bxi|\leq T^{\sigma_2}, 
\qquad
H(\bxi)\leq T.
\end{equation}
\end{definition}
Now we can describe what it means for 
$\cL$ to satisfy Hypothesis D.
\begin{definition}
\label{HypothesisD}
Let $\cL$ be as described at the beginning of this section. Hypothesis D asserts that there exist 
\[
c > 0 \quad \text{and} \quad
\gam \in \left[
\frac{n-d}{n}, 1 
\right)
\]
such that Hypothesis $D(c, \gamma(1+d\sig_2), \sig_2)$ holds for all
\[
\frac1{n-d} \le \sig_2 \le 1.
\]
\end{definition}

\begin{lemma}
Let $\cL$ be as described at the beginning of this section, and suppose Hypothesis $D(c, \sig_1, \infty)$ holds for some $c > 0$ and some
\[
\sig_1 \in \left[ 1, \frac{n}{n-d} \right).
\]
Then Hypothesis D holds.
\end{lemma}

\begin{proof}
Put
\[
\gam = \sig_1 \frac{n-d}{n} \in \left[ \frac{n-d}{n}, 1 \right),
\]
and let $\sig_2 \geq 1/(n-d)$. Then
\[
\gam(1+d\sig_2) \geq \frac{\sig_1(n-d)}n \left(
1+\frac{d}{n-d}
\right)
= \sig_1,
\]
so Hypothesis $D(c, \gamma(1+d\sig_2), \sig_2)$ holds.
\end{proof}

We now discuss some special cases in which the condition simplifies.

\begin{example} Consider the case $d = n - 1$ of affine hyperplanes. Then
$A = (a_1, \ldots, a_d)$ is a row vector of non-zero reals. Our Monge form is
\[
(t_1, \ldots, t_d, a_0 + a_1 t_1 + \cdots + a_d t_d),
\]
where $a_0 = b$. Kleinbock \cite{Kle2003} considered the case of strong extremality, namely $\psi(q) = q^{-\tau}$ for $\tau > 1$. By \cite[Corollary 5.7]{Kle2003}, it is necessary to assume that 
\[
\omega(A,b) := \inf\{\sig\geq 0:\max \{ \| q a_j \| : 0 \le j \le d \} \gg_\sig q^{-\sig} \} \leq n.
\]
In this case, since $n-d=1$ we always have $\sig_2=1$, and Hypothesis D reduces to having
\[
\omega(A,b) < n
\]
in any ordering of the standard basis.

Note that Kleinbock allows some of the $a_i$ to vanish. Our method is also robust enough to achieve this, but we defer the discussion until after the proof, see Remark \ref{KleinbockRemark}.
\end{example}

\begin{example}
In the case $(n,d) = (2,1)$ of planar lines, the assumption that \eqref{Cond: D} holds for some $\sig_1 < 2$ with $\sig_2 = 1$ matches Huang's threshold \cite[Theorem 2.4]{Hua2024}.
\end{example}

We now show that Hypothesis D is almost sharp. A vector $\bx \in \bR^n$ is \emph{very well multiplicatively approximable} (VWMA) if it satisfies the following equivalent conditions \cite[\S1]{KM1998}:
\begin{enumerate}[(i)]
\item For some $\tau > 1$, there are infinitely many $q \in \bZ$ such that
\[
\| q x_1 \| \cdots \| q x_n \| < |q|^{-\tau}.
\]
\item For some $\tau > 1$, there are infinitely many $\bm \in \bZ^n$ such that
\[
\| \bm \cdot \bx \| < H(\bm)^{-\tau}.
\]
\end{enumerate}

\begin{lemma}
Let $\cL$ be as described at the beginning of this section. Let $\gamma > 1$ and $\sig_2 \ge 1/(n-d)$. Assume that, for infinitely many positive integers $T$, there exists 
$\bxi = \bxi_T \in \bZ^{n-d} 
\setminus \{ \bzero \}
$
such that
\[
\max \{ \| \bxi \cdot \ba^{(1)} \|,
\ldots, \| \bxi \cdot \ba^{(d)} \|, \| \bxi \cdot \bb \| \} < T^{-\gam(1 + d\sig_2)} 
\]
and
\[
|\bxi|\leq T^{\sig_2}, \qquad H(\bxi) \leq T.
\]
Then all points on $\cL$ are VWMA, and in particular $\cL$ is not convergent Gallagher.
\end{lemma}

\begin{proof} Let $\bt \in \bR^d$ and 
\[
\bx = \begin{pmatrix}
\bt \\
A \bt + \bb
\end{pmatrix}.
\]
Put
\[
v = \frac{1+\gam}2 (1 + d\sig_2),
\]
and let $\eps \in (0,1)$ be sufficiently small in terms of $A$ and $\bt$. By assumption, for some arbitrarily large $T$, there exist $\bp \in \bZ^{d+1}$ and $\bq \in \bZ^{n-d}$ with
\[
0 < |\bq| \le T^{\sig_2},
\qquad
H(\bq) \le T
\]
and
\[
|p_j + \bq \cdot \ba^{(j)}| < \eps^2 T^{-v}
\qquad (1 \le j \le d + 1),
\]
where $\ba^{(d+1)} = \bb$.

Write $(A,\bb) = (a_{i,j})$ and $t_{d+1} = 1$. By the triangle inequality,
\[
\left|
\sum_{\substack{i \le n-d \\ j \le d+1}}
(p_j + q_i a_{i,j}) t_j 
\right| < \eps T^{-v}.
\]
With $\bm = (p_1,\ldots,p_d,q_1,\ldots,q_{n-d})$, we therefore have
\begin{align*}
|\bm \cdot \bx + p_{d+1}| &= 
\left|
\left( \sum_{j \le d+1} p_j t_j \right) +
\sum_{i \le n-d} q_i \left(b_i + \sum_{j \le d} a_{ij} t_j \right)
\right| \\ &= 
\left|
\sum_{\substack{i \le n-d \\ j \le d+1}}
(p_j + q_i a_{i,j}) t_j 
\right| < \eps T^{-v}.
\end{align*}
Furthermore, since $|\bq|\leq T^{\sig_2}$,
\[
H(\bm) \ll_{A,\bt} |\bq|^d H(\bq) \le T^{1+d\sig_2} =: T_0.
\]
With
\[
\tau = \frac{v}{1+d\sig_2} = \frac{1 + \gam}{2} > 1,
\]
we have
\[
\| \bm \cdot \bx \| < \eps T_0^{-\tau} \ll_{A,\bt} H(\bm)^{-\tau}.
\]
As $T_0$ can be arbitrarily large, we see that $\bx$ is VWMA.
\end{proof}

Next, we show that Hypothesis D is generic.

\begin{lemma}
\label{exceptional}
Let $1 \le d < n$ be integers, and let $\gamma \ge 1$. Let $\cW$ be the set of matrices 
\[
A = (\ba^{(1)},\ldots, \ba^{(d+1)}) \in [0,1]^{(n-d) \times (d+1)}
\]
such that there are infinitely many $\bq \in \bZ^{n-d}$ for which
\begin{equation}
\label{strange}
\max \{ \| \bq \cdot \ba^{(1)} \|, \ldots, \| \bq \cdot \ba^{(d+1)} \| \} < H(\bq)^{-\sig_1},
\end{equation}
where
\begin{equation}
\label{stranger}
|\bq| = H(\bq)^{\sig_2}, \qquad
\sig_1 = \gamma(1 + d\sig_2).
\end{equation}
Then
\[
\dimh(\cW) = (d+1)(n-d-1)+\frac{d+2}{1+\gam(d+1)}
\le (n-d)(d+1)-d.
\]
\end{lemma}

\begin{proof}
We begin with the upper bound
\[
\dimh(\cW) \le (d+1)(n-d-1)+\frac{d+2}{1+\gam(d+1)},
\]
adapting a classical covering argument \cite{DV1997}. Put
\[
m = d + 1, \qquad
m' = n - d.
\]
For $\bp \in \bZ^m$ and $\bzero \ne \bq \in \bZ^{m'}$, let
\[
R_{\bp, \bq} = \{ A \in \bR^{m' \times m}: \bq^T A + \bp = \bzero \}
\]
and, for $\del > 0$, let
\[
\Delta(R_{\bp,\bq}, \del) =
\{
A \in [0,1]^{m' \times m}: \dist(A, R_{\bp, \bq}) < \del
\}.
\]
Note that if \eqref{strange} holds then, for some $\bp \in \bZ^m$, we have
\[
\dist(A, R_{\bp, \bq}) < 
H(\bq)^{-\sig_1} |\bq|^{-1},
\]
and so $A \in \Delta(R_{\bp,\bq}, 
H(\bq)^{-\sig_1} |\bq|^{-1})$. Consequently, for any $Q \ge 2$,
\[
\cW \subseteq \bigcup_{|\bq| \ge Q}
\bigcup_{\bp \in \bZ^m} \Del(R_{\bp, \bq}, H(\bq)^{-\sig_1}
|\bq|^{-1}).
\]
Let
\begin{equation}
\label{tau}
\tau > m(m'-1) + \frac{1 + m}{1 + \gam m}.
\end{equation}
The implied constants below depend at most on $m,m',\tau$.

Given $\bq \ne \bzero$, let us write $\del = H(\bq)^{-\sig_1}
|\bq|^{-1}$. Then
\[
\# \{ \bp \in \bZ^m:
\Del(R_{\bp,\bq}, \del) \ne \emptyset \} \ll |\bq|^m.
\]
Now $\Del(R_{\bp,\bq}, \del)$ is covered by a collection $\cC_{\bp,\bq}$ of balls of radius $\del$, where
\[
\# \cC_{\bp, \bq} \ll
\del^{m(1-m')}.
\]
By \eqref{stranger}, we have
\[
\sig_2 = \frac{\log |\bq|}{\log H(\bq)}, \qquad
\delta^{-1} = |\bq|^{\gamma d+1+\gamma \sig^{-1}_2}.
\]
Now
\begin{align*}
\cH^\tau(\cW) &\ll \sum_{|\bq| \ge Q} |\bq|^m \del^{\tau - m(m'-1)} \\
&= \sum_{|\bq| \ge Q} |\bq|^{m + (\gam d + 1 + \gam \sig_2^{-1})(m(m'-1) - \tau)}
\\
&= |\bq|^{(mm' - \tau) (\gam d+1)-m\gam d}H(\bq)^{-(\tau - m(m'-1)) \gam},
\end{align*}
and so
\begin{align*}
&\cH^\tau(\cW) \\
&\ll \sum_{q_{m'} \ge Q} \: \sum_{q_1, \ldots, q_{m'-1} \le q_{m'}} q_{m'}^{(mm'-\tau)(\gam d+1)-m\gam d} (q_1 \cdots q_{m'})^{-(\tau - m(m'-1))\gam} \\
&\ll \sum_{q_{m'} \ge Q} q_{m'}^{(mm'-\tau)(\gam d+1)-m\gam d-(\tau-m(m'-1))\gam} (\log q_{m'} + q_{m'}^{1-(\tau - m(m'-1))\gam})^{m'-1}.
\end{align*}

Taking $Q \to \infty$ furnishes $\cH^\tau(\cW) = 0$, as long as
\begin{equation}
\label{taureq}
(mm'-\tau)(\gam d+1)-m\gam d-(\tau-m(m'-1))\gam < -1
\end{equation}
and
\begin{align*}
&(mm' - \tau)(\gam d+1) - m\gam d -(\tau-m(m'-1))\gam\\
&\qquad \qquad + (m'-1)(1-(\tau - m(m'-1))\gam) < -1.
\end{align*}
These inequalities hold as long as
\[
\tau >
\max \left\{ mm' + \frac{1-m(\gam d + \gam)}{1+\gam d +\gam},
m(m'-1) + \frac{m+m'}{1+\gam d+m'\gam}
\right \}.
\]
Since $m = d + 1$ and $m + m' - 1 = n$, we require that
\[
\tau > \max \left\{ mm' + \frac{1- \gam m^2}{1+\gam m},
m(m'-1) + \frac{1 + n}{1 + n \gam}
\right \},
\]
i.e.
\[
\tau > \max \left\{ m(m'-1) + \frac{1+m}{1+\gam m},
m(m'-1) + \frac{1 + n}{1 + \gam n}
\right \}.
\]
As $\gam \ge 1$ and $n \ge m$, the first term dominates. Thus, by \eqref{tau}, our requirement is met and indeed $\cH^\tau(\cW) = 0$. Therefore
\begin{align*}
\dimh(\cW) &\leq m(m'-1) + \frac{1+m}{1+\gam m} \\
&= (d+1)(n-d-1)+\frac{d+2}{1+\gam(d+1)}.
\end{align*}

\bigskip

The lower bound follows by considering only vectors of the form 
\[
\bq=(0,\dots,0,q).
\]
For such vectors,
\[
\sig_2 = 1, \qquad
\sig_1 = \gamma(1+d).
\]
Specialising
\[
(m,n,\boldsymbol \bet,\lam) = (1,d+1,\bzero,\gam(d+1))
\]
in \cite[Theorem 1]{Lev1998}, we obtain some $\cW' \subseteq \cW$ that is the Cartesian product of a set $\cW''\subseteq \mathbb{R}^{d+1}$ with $\mathbb{R}^{(d+1)(n-d-1)}$, where
\[
\dimh(\cW'')=\frac{d+2}{1+\gam(d+1)}.
\]
From here, we see that
\[
\dimh(\cW)\geq \dimh(\cW')=(d+1)(n-d-1)+\frac{d+2}{1 + \gam (d+1)}.
\]
\end{proof}

\subsection{Equidistribution on tori}
\label{equidistribution}

\begin{lemma}
\label{equi}
Let $A, \bb$ satisfy \eqref{Cond: D} whenever we have \eqref{whenever}, where
\[
c > 0, \qquad
1 \le \sig_1 \le n,
\qquad
\frac1{n-d} \le \sig_2 \le 1.
\]
Let
\begin{align*}
\cP &= \left \{ 
q_1 \ba^{(1)} + \cdots + q_{d+1} \ba^{(d+1)}\mmod\mathbb{Z}^{n-d}:\bq \in \bZ^{d+1}, \: 0 \le q_j \le Q - 1 \: \forall j
\right \} \\ &\subset \bT^{n-d},
\end{align*}
where $\ba^{(d+1)} = \bb$. Let $\cS$ be an axis-aligned rectangle in $\bT^{n-d}$, and let $Q \in \bN$. If
\[
\lam_{n-d}(\cS) \ge \eps Q^{\eps - 1/\sig_1}
\]
for some $\eps > 0$, and each side length of $\cS$ is at least $\eps \lam_{n-d}(\cS)^{\sig_2}$, then
\[
\# \cS \cap \cP \ll_{n,c,\eps} \lam_{n-d}(\cS) Q^{d+1}.
\]
\end{lemma}
By symmetry, the lemma also holds with
\[
\{ 
q_1 \ba^{(1)} + \cdots + q_{d+1} \ba^{(d+1)}:\bq \in \bZ^{d+1} \mmod \bZ^{n-d}, \: |q_j| \le Q \: \forall j \} \subset \bT^{n-d}
\]
in place of $\cP$.

\begin{proof}
We may assume that $Q$ is arbitrarily large in terms of $n,c,\eps$. Now 
consider the vectors 
\[
\bv_{\bq}=\sum_{j \le d+1} q_j \ba^{(j)}
\mmod \bZ^{n-d},
\]
as $\bq$ varies. 
Let $\bP$ be the probability measure obtained by placing a Dirac mass $Q^{-d-1}$ at $\bv_{\bq}$, for each $\bq$. 
For each $1\leq j\leq n-d$, let $r_j$ be the side length of $\cS$ in direction $j$. 
Let $\varpi$ be a bump function adapted to $\cS$, so that $\varpi \ge 1$ on $\cS$ and \mbox{$\| \varpi \|_1 \le 2\lam_{n-d}(\cS)$,} and the support of $\varpi$ is contained in an axis-aligned rectangle of length $2r_j$ in direction $j$, for each $j$.

By Plancherel, 
\begin{align*}
\# \cS \cap \cP &\le Q^{d+1} \int_{\bT^{n-d}} \varpi(\bx) \d \bP(\bx) = \sum_{\bxi \in \bZ^{n-d}} \hat \varpi(\bxi) \hat \bP(\bxi).
\end{align*}
The Fourier coefficients are given by 
\begin{align}
\notag
\hat \bP(\bxi) &= Q^{-d-1}\sum_{\bx \in \cP} e(- \bxi \cdot \bx) = Q^{-d-1} \sum_{\bq } e(- \bxi \cdot \bv_\bq)\\
\notag &=
Q^{-d-1} \sum_{q_1,\ldots,q_{d+1} =0}^{Q-1} e(-\langle \bxi, q_1 \ba^{(1)} + \cdots + q_{d+1} \ba^{(d+1)} \rangle ) \\
\label{PhatBound}
&= \prod_{j \le d+1} Q^{-1} \sum_{q = 0}^{Q - 1} e( -q \bxi \cdot \ba^{(j)} ) \ll \prod_{j \le d + 1} \min \{ 1,
(Q \| \bxi \cdot \ba^{(j)} \| )^{-1} \}.
\end{align}
By the rapid decay of $\hat \varpi$, 
\begin{align*}
\frac{\# \cS \cap \cP}{Q^{d+1}} 
\ll \lam_{n-d}(\cS) \sum_{ \substack{
\bxi \in \bZ^{n-d} \\
|\xi_j| \le Q^{\eps/n}/r_j \: \forall j }} 
|\hat \bP(\bxi)|.
\end{align*}

Now consider the points
\begin{equation}
\label{points}
(\bxi \cdot \ba^{(1)}, \ldots, \bxi \cdot \ba^{(d+1)}) \in \bT^{d+1},
\end{equation}
for non-zero vectors $\bxi \in \bZ^{n-d}$ with $|\xi_j| \le Q^{\eps/n}/r_j$ for all $j$. 
The assumption \eqref{Cond: D}, applied with
\[
T = \prod_{j \le n-d} (Q^{\eps/n}/r_j),
\]
implies the following gap principle: if $\bz$ is such a point, or the difference between two distinct such points, then
\[
\| \bz \|_{\bT^{d+1}} \gg \lam_{n-d}(\cS)^{\sig_1} Q^{-\eps \sig_1 (n-d)/n}.
\]
Since
$
\lam_{n-d}(\cS) \gg Q^{\eps - 1/\sig_1},
$
it follows that
\begin{equation}
\label{GapPrinciple}
\| \bz \|_{\bT^{d+1}} \gg Q^{\eps d/n - 1}.
\end{equation}

To finish, we cover $\bT^{d+1}$ by cubes of side length $Q^{\eps d/n - 1}$. By our gap principle \eqref{GapPrinciple}, any such cube contains at most $O(1)$ many of the points \eqref{points}. Thus, by \eqref{PhatBound},
\begin{align*}
\sum_{ \substack{
\bzero \ne \bxi \in \bZ^{n-d} \\
|\xi_j| \le Q^{\eps/n}/r_j \: \forall j }} |\hat \bP(\bxi)| 
&\ll \left( \sum_{q = 0}^Q \frac1{1 + qQ^{\eps d/n}} \right)^{d+1}
\ll 1.
\end{align*}
\end{proof}

\subsection{Target-hitting estimate}
\label{TargetHitting}
\begin{lemma}
\label{HitCount}
Let $\cR$ be an axis-aligned rectangle in $[0,1)^n$, with side length $r_j$ in direction $j$ for each $j$, where
\[
0 < r_1 \le \cdots \le r_n \le 1.
\]
Assume that $r_{d+1} =(r_{d+1} \cdots r_n)^{\sig_2}$. Let $A,\bb$ satisfy \eqref{Cond: D} whenever we have \eqref{whenever}, where
\[
c > 0, \qquad
1 \le \sig_1 \le n.
\]
Let $Q \in \bN$ and $\eps > 0$ be such that 
\[
\lam_n(\cR)^{\chi} \ge \eps Q^{\eps - 1/\sig_1},
\]
for $\chi = 1/(1+d\sig_2)$.
Then
\[
\lam_d( \{
\bt \in \fB: \exists q \in [Q,2Q] \quad (q \bt, q (A \bt + \bb))\mmod \mathbb{Z}^n \in \cR  \}) \ll_{n,c,\eps,A,T}
Q \lam_n(\cR),
\]
for any ball $\fB \subset \bR^d$ of radius
$T \ge 2$ centred at the origin. 
\end{lemma}

\begin{proof}
Define the projection
\begin{align*}
\pi: [0,1)^n &\to \bR^{n-d} \\
(\bx,\by) &\mapsto \by - A \bx.
\end{align*}
Then $\pi(\cR) \mmod \mathbb{Z}^{n-d} \subseteq \cS$, for some axis-aligned rectangle $\cS \subseteq \mathbb{T}^{n-d}$ where, for each coordinate $d+1 \le j \le n$, the side length of $\cS$ in direction $j$ is at least $r_j$ and $O_A(r_j)$.

For $q = Q, Q+1, \ldots, 2Q$, let
\[
\cL_q = \{ (q\bt, q(A\bt + \bb)): \bt \in \fB \} \subseteq \bR^n.
\]
Let $I_d$ be the $d$-dimensional identity matrix.
We cover $\cL_q$ by a union of small regions
\[
\wp = \left \{ 
\begin{pmatrix}
\bm \\ \bp
\end{pmatrix}
+ \begin{pmatrix}
I_d \\
A
\end{pmatrix} 
\bz: \bz \in [0,1)^d \right \},
\]
where $\bm \in \bZ^d$ and
\begin{align*}
\bp &= A\bm + q \bb
\\ &\in \tilde \cP :=
\left \{ 
q_1 \ba^{(1)} + \cdots + q_d \ba^{(d)} + q_{d+1} \bb:
-2TQ \le q_1, \ldots, q_{d+1} \le 2TQ
\right \}.
\end{align*}

Let $N$ be the total number of these small regions $\wp$ such that $\wp \mmod \mathbb{Z}^n$ intersects $\cR$,
as we range over $q=Q,Q+1,\ldots,2Q$. The key observation is that if $\wp$ intersects $\cR + \bZ^n$ then its starting point $\begin{pmatrix} \bm \\ \bp \end{pmatrix}$ lies in 
\[
(\{ \bzero \} \times \pi(\cR)) + \bZ^n.
\]
Therefore 
\[
N \ll_T \# ((\pi(\cR) \mmod \mathbb{Z}^{n-d}) \cap \cP)\le \# \cS \cap \cP,
\]
where
$\cP = \tilde \cP \mmod \bZ^{n-d}$.

Observe that
\[
r_1 \cdots r_n = (r_1 \cdots r_d) (r_{d+1} \cdots r_n) \leq (r_{d+1} \cdots r_n)^{d\sig_2} r_{d+1} \cdots r_n,
\]
whence
\[
r_{d+1} \cdots r_n \geq (r_1 \cdots r_n)^{\chi}.
\]
Note also that
\[
\frac1{n-d} \le \sig_2 \le 1.
\]
Since
\[
\lam_{n-d}(\cS) \gg r_{d+1} \cdots r_n \ge (r_1 \cdots r_n)^\chi \gg Q^{\eps - 1/\sig_1}
\]
and
\[
r_{d+1} = (r_{d+1} \cdots r_n)^{\sig_2} \gg \lam_{n-d}(\cS)^{\sig_2},
\]
we may apply Lemma \ref{equi}, giving
\[
N \ll \lam_{n-d}(\cS) Q^{d+1} \ll r_{d+1} \cdots r_n Q^{d+1}.
\]

For each $\wp$, the set of $\bt \in \fB$ such that
\[
(q\bt, q(A\bt + \bb))\mmod \mathbb{Z}^n \in (\wp \mmod \mathbb{Z}^{n}) \cap \cR 
\]
is contained in a union of $O_T(1)$ many rectangles with side lengths
\[
O(r_1/q), \ldots, O(r_d/q).
\]
Therefore
\begin{align*}
&\lam_d ( \{
\bt \in \fB: \exists q \in [Q,2Q] \quad (q \bt, q (A \bt + \bb)) \mmod \bZ^n \in \cR \}) \\
&\le Q^{-d} r_1 \cdots r_d  N
\ll Q \lam_n(\cR).
\end{align*}
\end{proof}

\subsection{Completing the proof of Theorem \ref{FlatConv}}
\label{completing}

In this subsection, we prove Theorem \ref{FlatConv}. Let 
\[
\gam \in \left[ \frac{n-d}{n}, 1 \right)
\]
and $c > 0$ be as in Hypothesis D. Let
\[
\cG = \left \{ (\gam(1+d\sig_2), \sig_2):
\frac1{n-d} \le \sig_2 \le 1 \right \}.
\]
By assumption, if $(\sig_1, \sig_2) \in \cG$ then $D(c, \sig_1, \sig_2)$ holds.
Let
\[
0 < \eps < 1 - \gam.
\]

Let $\psi: \bN \to [0,1)$ be monotonic, with
\[
\sum_{q=1}^\infty \psi(q) (\log q)^{n-1} < \infty.
\]
We claim we may assume that for $(\sig_1,\sig_2) \in \cG$, 
\[
\psi(q)^{1/(1+d\sig_2)} > q^{(\eps - 1)/\sig_1} \qquad (q \in \bN).
\]
To see this, note that if $(\sig_1, \sig_2) \in \cG$ then
\[
(1 - \eps) \frac{1 + d \sig_2}{\sig_1} = \frac{1 - \eps}{\gam} > 1.
\]
Hence, replacing $\psi(q)$ by $\tilde \psi(q) := \psi(q) + q^{(\eps - 1)/\gamma}$ yields a monotonic function such that
\[
\cW_n^\times(\tilde \psi) \supseteq
\cW_n^\times(\psi),
\qquad
\sum_{q=1}^\infty \tilde \psi(q) < \infty.
\]
This justifies the claim.

\bigskip

Let $\mu_d$ be the induced Lebesgue measure on $\cL$. Let $T \ge 2$, and let $\fB \subset \bR^d$ be the ball of radius $T$ centred at the origin. 
By symmetry, it suffices to prove that
$
\mu_d(\cW) = 0,
$
where $\cW$ is the set of 
$
\bx \in \fB \times \bR^{n-d}
$
such that 
\[
\| qx_1 \| \cdots \| qx_n \| < \psi(q), \qquad
\| qx_1 \| \le \cdots \le \| qx_n \| 
\]
has infinitely many solutions $q \in \bN$.
For $Q = 1, 2, 4, 8, \ldots$, we consider the hyperbolic region
\[
\cT_Q = \{ \bx \in \bT^n: \| x_1 \| \cdots \| x_n \| \le \psi(Q), \quad
\| x_1 \| \le \cdots \le \| x_n \| \}.
\]
We have
\[
\cW \subseteq \limsup_{Q \in \{ 1, 2, 4, 8, \ldots \}} \tilde \cE_Q,
\]
where
\[
\tilde \cE_Q = \{ ( \bt, A\bt + \bb): \bt \in \cE_Q \},
\]
wherein
\[
\cE_Q = \bigcup_{Q \le q \le 2Q} \{ \bt \in \fB: (q\bt, q(A\bt + \bb)) \mmod \bZ^n \in \cT_Q \}.
\]

We can cover $\cE_Q$ by $O((\log Q)^{n-1})$ many sets of the form
\[
\cE = \cE(\cR) = \bigcup_{Q \le q \le 2Q} \{ \bt \in \fB: (q\bt, q(A\bt + \bb)) \mmod \bZ^n \in \cR \}.
\]
Here $\cR \subseteq \bT^n$ is an axis-aligned rectangle, with side length $r_j$ in direction $j$ for each $j$, where
\[
0 < r_1 \le \cdots \le r_n \le 1,
\]
and
$
\lam_n(\cR) \asymp \psi(Q).
$
We apply Lemma \ref{HitCount} with $(\sig_1, \sig_2) \in \cG$ given by
\[
r_{d+1} = (r_{d+1} \cdots r_n)^{\sig_2}, \qquad
\sig_1 = \gam(1 + d \sig_2).
\]
This furnishes
$
\lam_d(\cE) \ll Q \lam_n(\cR) \ll Q \psi(Q),
$
and so
\[
\lam_d(\cE_Q) \ll Q \psi(Q) (\log Q)^{n-1}.
\]
The Cauchy condensation test now delivers
\[
\sum_{Q=1,2,4,8,\ldots} 
\mu_d(\tilde \cE_Q) \ll
\sum_{Q=1,2,4,8,
\ldots} \lam_d(\cE_Q) < \infty.
\]
Thus, by the first Borel--Cantelli lemma,
\[
\mu_d(\cW) \le \mu_d \left( \limsup_{Q \in \{ 1,2,4,8,\ldots\} } 
\tilde \cE_Q \right) = 0,
\]
completing the proof of Theorem \ref{FlatConv}.

\begin{rem} Our approach can also handle simultaneous approximation. The outcome is similar to Huang's Khintchine-type convergence theory for affine subspaces \cite{Hua2024}.
\end{rem}

\begin{rem}
\label{KleinbockRemark}
Let us revisit the case $d = n-1$ of affine hyperplanes, to discuss the situation where some of the parameters vanish. Suppose $\cL$ has Monge form
\[
(t_1, \ldots, t_d, \ba \cdot \bt + b),
\]
and that
\[
a_1 \cdots a_r \ne 0, \qquad
a_{r+1} = \cdots = a_d = 0.
\]
Then, upon changing the order of the coordinates, our affine hyperplane $\cL$ has the form $L \times \bR^{n-r-1}$, where $L$ is the affine hyperplane in $\mathbb{R}^{r+1}$ with Monge form
\[
(t_1, \ldots, t_r,
a_1 t_1 + \cdots + a_r t_r + b).
\]
The target-hitting problem is then reduced to the corresponding problem in $\mathbb{R}^{r+1}.$ There, we find that if
\[
\|q(a_1,\dots,a_r,b)\| \gg_\sig q^{-\sigma}
\qquad (q \in \bN)
\]
holds for some $\sigma < r+1$, then the corresponding target-hitting estimate holds. This gives rise to the convergence Gallagher property for $L\times \bR^{n-r-1}$ and hence $\cL$. The condition that $\sig < r+1$ matches Kleinbock's \cite[Corollary 5.7]{Kle2003}.
\end{rem}

\section{Proof of Theorem \ref{thm: main curves}}\label{sec: proof of curves}
\subsection{A model class of curves}
\label{subsec model class}
For sufficiently smooth curves in 
$\mathbb{R}^n$ given by \eqref{eq: curve}, the non-degeneracy condition is equivalent to the invertibility of the Wronskian of $f_1', f_2', \ldots, f_n'$, given by
\begin{equation}
\label{def wronsk}
W_{\bf}(\dx):=
\begin{pmatrix}
f_1'(\dx) & f_2'(\dx) & \ldots & f_n'(\dx)\\
f_1^{(2)}(\dx) & f_2^{(2)}(\dx) &\ldots & f_n^{(2)}(\dx)\\
\vdots & & & \vdots\\
f_1^{(n)}(\dx) & f_2^{(n)}(\dx) &\ldots & f_n^{(n)}(\dx)
\end{pmatrix},
\end{equation}
uniformly over $\dx\in U$, as Proposition 2.13 of 
Beresnevich and Yang \cite{BY2023} shows.

In the discussion ahead, we shall work with smooth variants of counting functions for rational points near $\cC$. 
To pass from the usual counting function to these smooth variants, we need flexibility in the parametrisation 
of the curve $\cC$. More specifically, for $\bdel\in \cD(t)$, let $j\leq n$ be such that $\delta_{\textrm{min}}=\delta_{j}$. 
Then we mandate that for any $1\leq i\leq n$, the curve $\cC$ be expressible in the Monge form 
$$\cC = \{ \dx \mathbf{e}_{i}+\mathbf{G}_{i}(\dx): \dx\in I_i\subset U_i, \quad \mathbf{G}_{i}: U_i\to \mathbf{e}_{i}^{\perp} \textrm{ is smooth}\},$$
where $I_i$ is a closed interval contained in an open set 
$U_i\subset \mathbb{R}$. In other words,  for any $i$ it should be possible to change variables 
from $x$ to $f_i(x)$. For this, we need non-zero lower bounds on 
$\inf_{\dx\in \I}|f_i'(\dx)|$. While they may not necessarily be available on the entire interval $\I$, for every $\rho>0$, 
we can still extract a set $\I_{\rho}\subset \I$ 
where this is true and the measure $\lam_1(\I\setminus \I_{\rho})$ is small.
Such reparametrisations are almost surely possible, as we see next.

\begin{lemma}
\label{lem red model setting}
Let $U\subset\mathbb{R}$ be an open interval containing a closed interval $\I$. Let $\bfg = (g_1, \ldots, g_n): U\rightarrow \R^n$ be a smooth, $l$-non-degenerate map. For every sufficiently small $\rho >0$, there exists a set $\I_{\rho}\subset \I$ which can be decomposed as
\begin{equation}
\label{eq Ieps sum}
\I_{\rho}=\bigsqcup_{j=1}^{K} I_j,
\end{equation}
where $I_j$ are closed intervals, with the following properties:
\begin{enumerate}[(i)]
\item
For all $1\leq i\leq n$ and any $\dx\in I_j$, we have
\begin{equation*}
    \label{eq I eps deriv}
    \rho \leq |g_{i}'(\dx)| \leq C_{\bf}.
\end{equation*}
\item
The set $\I_{\rho}$
is nearly all of $\I$, in the sense that
\begin{equation*}
    \label{meas I eps}
    \lam_1(\I\setminus \I_{\rho})\ll_\bf \rho^{1/(n-1)}.
\end{equation*}
\end{enumerate}

In particular, for each $1\leq j\leq K$, there exists an open set $U_j\subset U$ containing $I_j$ such that $g_i$ restricted to $U_j$ 
is a smooth, invertible map with inverse $g_i^{-1}$ for any $1\leq i\leq n$, and
\begin{equation*}
\label{eq inv deriv est}
\sup_{y \in U_j}|\left(g_i^{-1}\right)^{(N)}(y)| \ll_N \rho^{-N}
\end{equation*}
for each $N\geq 1$. 
\end{lemma}

The proof of this is a standard application of harmonic analysis, specifically sub-level estimates of van der Corput type. We defer it to Appendix~\ref{app proof lemma}.

\begin{rem}
Fixing $\rho>0$, let $I_j$ and $\I_\rho=\bigsqcup_{j=1}^{K} I_j$
be as in Lemma \ref{lem red model setting} with $\bfg=\bf$.
Suppose we had
established Theorem \ref{thm: main curves}
on each $I_j$ separately and shown that
$$\mu_1(\cC_j \cap\cW^{\times}_n(\psi))=0,$$
where $\cC_j:=\{\mathbf{f}(\dx): \dx\in I_j\}$.
Then it follows from \eqref{meas I eps} that 
$$
\mu_1(\cC \cap\cW^{\times}_n(\psi))
\ll_\bf \rho^{1/(n-1)}
+ \sum_{1\leq j\leq K} \mu_1(\cC_j \cap\cW^{\times}_n(\psi))
=\rho^{1/(n-1)}.
$$
Theorem \ref{thm: main curves}
then follows readily by sending $\rho\to 0^+$. 
Therefore, without loss of generality, 
we can and shall assume henceforth that 
\begin{equation}
\label{eq rho assump}
\min_{1\leq i\leq n}\inf_{\dx\in U} |f_i'(\dx)|>\rho, 
\qquad  \max_{1\leq i\leq n}\sup_{y\in f_i(U)}|
\left(f_i^{-1}\right)^{(N)}(y)|
\ll_N \rho^{-N}.
\end{equation}
Letting $f_0$ denote the identity function on $\R$, we define
\begin{equation}
\label{def L and Crho}
L= 10 C_\rho, \quad
\mathrm{where}
\quad C_\rho=
\max_{0\leq i, j\leq n }
\sup_{y\in f_i(U)}
|(f_j\circ f_i^{-1})'(y)|.
\end{equation}
These quantities plays an key role 
in smoothing the counting functions later. 
We record them here to make their dependency on 
$\rho, \bf$ transparent and do the same for terms
depending on them. 
For instance, the dyadic cut-off parameter
$\Xi$, to be introduced in due course,
is to be chosen in terms of $L$, and all constants below are allowed to depend on $\rho$ and $L$.
\end{rem}

\subsection{Simplification to \texorpdfstring{$\psi$}{psi} decreasing with a rate}
\label{subsec psi simp}
To work with multiplicative 
Diophantine approximations effectively, 
we make a standard reduction 
in this section, repackaging the
multiplicative approximation 
conditions into a collection of 
simultaneous approximation conditions 
involving 
anisotropically-scaled dyadic rectangles. 
Concretely, we cover the 
hyperbolic region by a collection of dyadic 
rectangles.
Naturally, the cardinality of those rectangles
is an appropriate logarithm power. 
The cuspidal regions 
play a special role, and we shall 
now proceed to show how to truncate
their contribution suitably. 
\begin{rem}
The series \eqref{eq: series of measures} converges for
$$
\psi_{\theta}(q)= \frac{1}{q^{1 + \theta}}
$$
whenever $\theta > 0$.
We claim that, 
for a small but fixed $\theta>0$, 
we can assume  
\begin{equation}\label{eq lower bound psi}
\psi(q) \geq \frac{1}{q^{1 + \theta}},
\end{equation}
Indeed, otherwise we replace $\psi$
with $\psi'(q):=\max (\psi(q),\psi_{\theta}(q))$
which still satisfies the convergence 
condition \eqref{eq: series of measures}.
Since $\cW^{\times}_n(\psi) \subseteq 
\cW^{\times}_n(\psi')$, it suffices to prove
Theorem \ref{thm: Gallagher conv}
with $\psi'$ 
in place of $\psi$. This confirms the claim.
\end{rem}

By Cauchy's condensation test,
\begin{equation}
\sum_{t\geq1}2^{t}t^{n-1}\psi(2^{t})<\infty.\label{eq: divergence dyadic}
\end{equation}
Given $q\in [2^t,2^{t+1}]$, 
we put each of the $n$ factors in 
the product appearing in
\eqref{def mult app cond}
into a dyadic range.
To this end, we define 
\begin{equation*}
\label{def dyadic ranges unpruned}
\widetilde{\cD}(t):=\left\{(2^{-z_{1}},\ldots,2^{-z_{n}})
:\begin{array}{c}
z_{1},\ldots,z_{n}\in\Z_{\geq 1},\\
\prod_{i\leq n} 2^{-z_i} \leq \psi(2^{t})
\leq 2
\prod_{i\leq n} 2^{-z_i}
\end{array}\right\}
\end{equation*}
where $\Z_{\geq m} = \Z\cap [m,\infty)$ for any $m \ge 1$.
It should be no surprise that the main
difficulty lies in working with 
\begin{equation*}
\label{def dyadic ranges}
\cD(t):=\left\{(2^{-z_{1}},\ldots,2^{-z_{n}})
:\begin{array}{c}
z_{1},\ldots,z_{n}\in\Z_{\geq \Xi},\\
\prod_{i\leq n} 2^{-z_i} \leq \psi(2^{t})
\leq 2
\prod_{i\leq n} 2^{-z_i}
\end{array}\right\},
\end{equation*}
where $\Xi \geq 1$ is a large real number. 
For concreteness, if $L$ as in \eqref{def L and Crho} 
is given, then 
\begin{equation}\label{def Xi}
\Xi:= 10 + \frac{\log L}{\log 2}
\end{equation}
is an acceptable choice. Morally speaking, 
if any of the $z_i$ is smaller than a fixed constant
then we are dealing with a 
lower-dimensional approximation problem.
The multiplicative convergence theory for planar curves, i.e. the case $n=2$, was treated
by Badziahin and Levesley 
\cite[Theorem 1]{BL2007}.
We will induct 
on the dimension to 
handle such $\bz$.

\begin{rem}
Let $(2^{-z_{1}},\ldots,2^{-z_{n}})\in
\widetilde{\cD}(t)$. In view of 
\eqref{eq lower bound psi},
we have 
\begin{equation}
\label{ziAlt}
\prod_{i\leq n} 2^{-z_i}
\geq 
\frac{1}{2}\psi(2^{t}) \geq \frac{1}{2}2^{-(1+\theta)t}
\geq 2^{-(1+\theta)t-1}.
\end{equation}
Thus, 
\begin{equation}
    \label{eq up bd zi}
    \max_{1\leq i\leq n} z_i \leq
    \sum_{1\leq i\leq n} z_i \leq (1+\theta)t+1,
\end{equation}
which will be useful later.
Consequently,
\begin{equation}
\#\cD(t)\leq \# \widetilde{\cD}(t)
\ll t^{n-1}
\label{eq: size dyadic exponents}.
\end{equation}
These cardinality bounds are
in tune with the term $t^{n-1}$ occurring 
in \eqref{eq: divergence dyadic}.
\end{rem}

\subsection{A Counting Problem}
\label{sub sec counting prob}
Theorem \ref{thm: main curves} 
will be a consequence of an estimate 
for the number of 
rational points contained in certain 
anisotropic neighbourhoods 
of subsets of $\cC$. 
For $t\in \mathbb{N}$, 
$\bdel\in (0, 1)^{n}$ 
and $\cG\subseteq \I$, 
we define the counting function
\begin{align*}
 \cR(\cG, \bdel, t) =
 \# \left\{(\ba, q)\in \mathbb{Z}^{n+1}:
2^t \leq q<2^{t+1},
\quad \exists \dx\in \cG \textrm{ with } \eqref{eq proximity cond}
\right\},   
\end{align*}
where \eqref{eq proximity cond}
is the collection of the $n$
inequalities
\begin{equation}\label{eq proximity cond}
    \Big\vert f_i(\dx)-\frac{a_i}{q}\Big\vert
<\frac{\delta_i}{q} \textrm{ for } 1\leq i\leq n.
\end{equation}
This is to say that $\cR(\cA, \bdel, t)$ counts rational points with denominator of size $2^t$ contained in a non-isotropic neighbourhood of the portion of the curve $\cC$ parametrised over a subset $\cG$ of $\I$. 

Theorem \ref{thm: main curves} will follow from the following theorem which achieves two objectives: count rational points in a neighbourhood of a `good' portion of the curve $\cC$, and estimate the measure of the `bad' portion where such a result is out of reach, depending on the parameters $\bdel$ and $t$.
Previously in \cite{SST}, a similar strategy 
was used for the Khintchine-type convergence theory. Here, the neighbourhoods 
of the curve in which the rational points need to be counted are boxes with widely different side lengths, 
rather than hypercubes. This introduces new difficulties and forces us to revisit the method and
introduce a couple of new ideas. Firstly, we discretise the possible side lengths into 
logarithmically many dyadic choices, possibly reparametrising so that the first coordinate is
placed where the smallest side length is. Secondly, the usage of the non-divergence estimates 
will need to be sensitive to different parameters (in fact the size of the smallest side).
Thirdly, we carefully inflate the dyadic scales --- should they be too small --- 
to a size in which we can just handle the 
dyadic rectangles. 

\begin{theorem}\label{thm: prob count}
Let $\rho>0$ be as in \eqref{eq rho assump}. For every integer $t\geq 1$ and any $\bdel\in \cD(t)$, there exists a set $\cA_\bdel\subset \I$
such that: 
\begin{enumerate}[(i)]
\item
\begin{equation}
\label{eq: convergence of measures}
\sum_{t\geq 1}\sum_{\bdel\in \cD(t)} \lam_1(\cA_{\bdel}) < \infty.
\end{equation}
\item
Writing 
\begin{equation*}
\label{def bdel times}
\bdel^{\times} = \frac{\prod_{1\leq i\leq n} \delta_i}{\min_{1\leq i\leq n} \delta_i},
\end{equation*}
we have
\begin{equation*}
\label{eq Bdel count}
\#\cR(\I\setminus \cA_{\bdel}, 4\bdel, t) \ll_{\rho} 2^{2t}\bdel^{\times},
\end{equation*}
where the implicit constant is uniform in $\bdel=(\delta_1, \ldots, \delta_n) \in \cD(t)$ and $t\geq 1$.
\end{enumerate}
\end{theorem}

\subsection{Proof of Theorem \ref{thm: main curves}
assuming Theorem \ref{thm: prob count}}
\label{subsec conv mod counting}
We need two simple lemmas.
\begin{lemma}
\label{le surjection}
Let $[[u]]:=[1,u] \cap \Z$, and $[[u,v]]:=[u,v] \cap \Z$.
If $u_1,\ldots,u_n\in \bN$, then 
\begin{align*}
[[u_1]] \times \cdots \times [[u_n]] &\rightarrow 
[[n,u_1 + \cdots +u_n]] \\
(z_1,\ldots,z_n) &\mapsto z_1 + \cdots + z_n
\end{align*}
is surjective.
\end{lemma}

\begin{proof}
We omit the proof as it is straightforward.
\end{proof}

Next, we need a lemma to compare the lim sup sets of 
$\cD(t)$ and $\widetilde{\cD}(t)$.
\begin{lemma}\label{lemma: reduction to dyadic ranges}
For
$\bdel=\bdel(t)\in \widetilde{\cD}(t)$, let
$\cW_{\bf}^{\times}(\psi;\bdel)$
be the set of $x\in \I$ such that there exist
$q\in \mathbb{N}$ 
with $2^{t}\leq q< 2^{t+1}$ and 
$a_1, \ldots, a_n\in \mathbb{Z}$ satisfying
\begin{equation}
\label{eq: dyadic delta inequality}
\Big\vert f_i(\dx)-\frac{a_i}{q}\Big\vert
<\frac{4\delta_i}{q}
\quad \textrm{for} \quad 1\leq i\leq n.
\end{equation}
Given $\cS\subseteq \widetilde{\cD}(t)$,
put 
$$
\cW_{\bf,\cS}^{\times}(\psi)=
\bigcup_{\bdel \in \cS} \cW_{\bf}^{\times}(\psi;\bdel).
$$
\begin{enumerate}[(a)]
\item We have  
\begin{equation}
\label{eq limsup inclusion}
\cC\cap \cW_n^{\times}(\psi) \subseteq 
\limsup_{t\rightarrow \infty}
\cW_{\bf,\widetilde{\cD}(t)}^{\times}(\psi).
\end{equation}
\item Suppose
Theorem \ref{thm: main curves}
is known to hold in $\R^{n-1}$.
Then replacing the dyadic ranges
$\widetilde{\cD}(t)$ with
the restricted ranges $\cD(t)$
on the right-hand side of 
\eqref{eq limsup inclusion}
does not change the measure, i.e.
\begin{equation*}
\label{eq limsup difference}
\lam_1\Big( \limsup_{t\rightarrow \infty}
\cW_{\bf,\widetilde{\cD}(t)}^{\times}(\psi)
\Big)
= 
\lam_1\Big(
\limsup_{t\rightarrow \infty}
\cW_{\bf,\cD(t)}^{\times}(\psi)
\Big).
\end{equation*}
\end{enumerate}
\end{lemma}

\begin{proof}
\begin{enumerate}[(a)]
\item
Let $\bf(\dx)\in \cW_n^{\times}(\psi)$.
If any $f_i(\dx)$ is a rational number, 
then we can pick $z_i$ arbitrarily large to confirm the desired claim that $$\bf(x) \in \limsup_{t\rightarrow \infty}
\cW_{\bf,\widetilde{\cD}(t)}^{\times}(\psi).$$
Therefore, we assume 
$\bf(\dx) \in (\R\setminus \Q)^n$.
Observe that some
strictly increasing sequence of positive integers $t_1, t_2,\ldots$ satisfies
$2^{t_j} \leq q_j < 2^{t_j+1}$
and
$$
\prod_{1\leq i\leq n} \|q_j f_i(\dx)\|\leq \psi(q_j)
\leq \psi(2^{t_j}),
$$
for some $q_j \in \bZ$ and all $j$.
Take $U_j\geq 1$ to be the integer with 
$2^{-(U_j+1)}\leq \psi(2^{t_j}) <2^{-U_j}$.
Then
$$
\prod_{1\leq i\leq n}\Vert q_j f_i(\dx)\Vert
\leq 2^{-U_j}.
$$

Since
$f_i(\dx)$ is irrational, we have
$ \Vert q_j f_i(\dx)\Vert >0$.
Let $u_{j,i} \geq 1$ be an integer
satisfying
$$
\Vert q_j f_i(\dx)\Vert \in 
(2^{-(u_{j,i}+1)},2^{-u_{j,i}}].
$$
By Lemma \ref{le surjection},
we can find integers $1\leq z_{j, i} \leq u_{j, i}+1 $ such that
$$z_{j, 1} +\cdots + z_{j, n} = U_j+1.$$
Then
$$\Vert q_j f_i(\dx)\Vert\leq 2^{-z_{j, i}+1}$$
implies that
$\bff(\dx)$ satisfies 
\eqref{eq: dyadic delta inequality} 
with $\bdel_j=(2^{-z_{j, 1}},\ldots, 2^{-z_{j, n}})$.
Further, $$\prod_{1\leq i\leq n}2^{-z_{j, i}}
= 2^{-U_j-1 } \leq  \psi(2^{t_j})< 2^{-U_j}=2\prod_{1\leq i\leq n}2^{-z_{j, i}},$$
whence $\bdel_j \in \widetilde{\cD}(t_j)$. 
\item 
Let
\begin{align*}
\cY & =
\Big(\limsup_{t\rightarrow \infty}
\cW_{\bf,\widetilde{\cD}(t)}^{\times}(\psi)
\Big)
\setminus
\Big(\limsup_{t\rightarrow \infty}
\cW_{\bf,\cD(t)}^{\times}(\psi)
\Big) \\
& \subseteq \limsup_{t\rightarrow \infty}
\Big(\cW_{\bf,
\widetilde{\cD}(t)}^{\times}(\psi)
\setminus
\cW_{\bf,\cD(t)}^{\times}(\psi)\Big).
\end{align*}
Denote by $\bG_i: \I \rightarrow \R^{n-1}$ 
the function obtained 
by deleting the $i^{th}$ component function
of $\bf$.
Let $\widetilde{\cD}_i(t)$
be the set of vectors $\bdel\in
\widetilde{\cD}(t)$
with $i^{th}$ component deleted.
Notice that
$$
\cY \subseteq 
\limsup_{t\rightarrow \infty}
\Big(
\bigcup_{i\leq n}\cW_{\bG_i,
\widetilde{\cD}_i(t)}^{\times}(2^{10+\Xi}\psi)
\Big)
=
\bigcup_{i\leq n}
\limsup_{t\rightarrow \infty}
\cW_{\bG_i,
\widetilde{\cD}_i(t)}^{\times}(2^{10+\Xi}\psi).
$$
Abbreviate by $\cC_i\subseteq \R^{n-1}$ 
the curve immersed by $\bG_i$.
To investigate the right-hand side, note
$$
\cW_{\bG_i,
\widetilde{\cD}_i(t)}^{\times}(2^{10+\Xi}\psi)
\subseteq 
\cC_i \cap\cW^{\times}_{n-1}(2^{10+\Xi}\psi).
$$
Each $\cC_i$ is non-degenerate.
In view of 
$$
\sum_{q\geq 1} \psi(q) (\log q)^{(n-1)-1}
\ll \sum_{q\geq 1} \psi(q) (\log q)^{n-1}<\infty,
$$
we deduce from 
Theorem \ref{thm: main curves} in $\R^{n-1}$ that 
\mbox{$\cC_i \cap\cW^{\times}_{n-1}(\psi)$} 
is a null set. 
\end{enumerate}
\end{proof}
We are now ready to prove our main result.
\begin{proof}[Proof of Theorem \ref{thm: main curves}
assuming Theorem \ref{thm: prob count}]
In light of Lemma \ref{lemma: reduction to dyadic ranges} (a), it suffices to show that
$$
\lam_1\Big(
\limsup_{t\rightarrow \infty}
\cW_{\bf,\widetilde{\cD}(t)}^{\times}(\psi)
\Big)=0.
$$
We argue by induction on the ambient dimension $n$.
The base case, $n=2$, was handled
by Badziahin and Levesley 
\cite[Theorem 1]{BL2007}. For $n\geq 3$, the induction assumption on $n-1$ allows us to 
use item (b) of 
Lemma \ref{lemma: reduction to dyadic ranges}.
Therefore, it is enough to establish
$$
\lam_1\Big(
\limsup_{t\rightarrow \infty}
\cW_{\bf,\cD(t)}^{\times}(\psi)
\Big)=0.
$$
To show this, let $\nu>0$ be given. 
For $t\in \mathbb{N}$ and $\bdel\in \mathcal{D}(t)$, let $\cA_{\bdel}$ be the exceptional set guaranteed by Theorem \ref{thm: prob count} and set $\cB_{\bdel}:=\I\setminus \cA_{\bdel}$. By Theorem \ref{thm: prob count}, we have  
$$\#\cR(\cB_{\bdel}, 4\bdel, t) \ll
2^{2t}\bdel^{\times}$$
uniformly in $\bdel\in \cD(t)$ and $t\geq 1$.

Let  
$$
\cP_{\bdel} = 
\bigcup_{(\ba,q)\in \cR(\cB_{\bdel}, 4\bdel, t)} 
\Big\{ \dx \in \cB_\bdel:
\Big\vert f_i(\dx)-\frac{a_i}{q}\Big\vert
<\frac{4\delta_i}{q}
\quad \text{ for } 1\leq i\leq n \Big\}.
$$
For any integer $T\geq 1$, we have
$$
\cE:= 
\limsup_{t\rightarrow \infty}
\cW_{\bf,\cD(t)}^{\times}(\psi)
\subseteq \bigcup_{t\ge T} \bigcup_{\bdel \in {\cD}(t)}\left(\cA_{\bdel} \cup \cP_{\bdel}\right).
$$
Hence, 
\begin{equation}
    \label{eq excep A C}
    \lam_1(\cE)\leq 
\sum_{t\geq T} \sum_{\bdel\in \cD(t)}\lam_1 (\cA_{\bdel})
+ 
\sum_{t\geq T} \sum_{\bdel\in \cD(t)}\lam_1 (\cP_{\bdel}).
\end{equation}
By \eqref{eq: convergence of measures}, 
\begin{equation}
\label{eq A small}
\lim_{T\to \infty}\,\sum_{t\geq T}\sum_{\bdel\in \cD(t)} \lam_1 (\cA_\bdel)=0.
\end{equation}

Let $\delta_{\textrm{min}}=\min_{1\leq i\leq n} \delta_i$. We observe $\mu_1 (\cP_\bdel)$
is at most $\#\cR(\cB_{\bdel},4\bdel, t)$ times
\begin{align*}
    & 
\max_{(\ba,q)\in \cR(\cB_{\bdel},4\bdel, t)}
\lam_1 \Big(
\Big\{ \dx \in \I:
\Big\vert f_i(\dx)-\frac{a_i}{q} \Big\vert 
<\frac{4\delta_i}{q} 
\textrm{ for all } 1\leq i\leq n 
\Big\}
\Big)
\ll \frac{\delta_{\textrm{min}}}{2^t}.
\end{align*} 
Theorem \ref{thm: prob count} ensures that
$$\#\cR(\cB_{\bdel},4\bdel, t)\ll 
2^{2t}\bdel^{\times}= 
\bdel^{\times},$$
whence
$$
\sum_{\bdel \in \cD(t)}\lam_1 (\cP_{\bdel}) \ll
\sum_{\bdel \in \cD(t)}
2^{2t}
\bdel^{\times} 
\frac{\delta_{\textrm{min}}}{2^{t}}
=
\sum_{\bdel \in \cD(t)}
2^{t}
\prod_{1\leq i\leq n} \delta_i
\ll 
\sum_{\bdel \in \cD(t)}
2^{t} \psi(2^t).
$$
Thus, by \eqref{eq: size dyadic exponents},
$$
\sum_{\bdel \in \cD(t)}\lam_1 (\cP_{\bdel})
\ll t^{n-1}2^{t} \psi(2^t).
$$
Summing over $t\geq T$, taking the limit $T\to \infty$, and using \eqref{eq: series of measures} gives
$$
\lim_{T\to \infty }\sum_{t\geq T} \sum_{\bdel \in \cD(t)}\lam_1 (\cP_{\bdel}) 
\ll \sum_{t\geq T}  
t^{n-1}2^{t} \psi(2^t)=0.
$$
Combining this with \eqref{eq excep A C} and \eqref{eq A small} delivers
$\lam_1(\cE)=0$.
\end{proof}

\subsection{Smooth counting functions}
\label{subsec smooth cutoff}
To establish Theorem \ref{thm: prob count}, 
we smooth our counting functions.
We introduce smooth cut-offs 
$$
\ome,w: \bR \rightarrow [0,1]
$$  
such that:
\begin{itemize}
\item 
The support of $\ome$ is contained in $[1/2,4]$, and $\ome(x)=1$ for $x\in [1, 2]$.
\item 
The support of
$w$ is contained in 
$(-2, 2)$, we have $w(x)=1$ for 
$x \in [-1, 1]$, and $w(x)=w(-x)$ 
for all $x\in \mathbb{R}$.
\end{itemize}
We shall also require a family of smooth weights adapted to $\cG$.

\begin{definition}
Let $\cG\subset \I$ and $\kappa\in (0, 1)$ be such that the $\kappa$ neighbourhood of $\cG$ is contained in $U$, where $U$ is as in
\eqref{eq: curve}.
A smooth weight $\Omega:\mathbb{R}\to [0, 1]$ is \emph{$(\cG, \kappa)$ adapted} if its support is contained in $U$ and it is identically $1$ on the $\kappa$ neighbourhood of $\cG$.
\end{definition}

Let $t\in \mathbb{N}$ and $\bdel=(\delta_1, \ldots, \delta_n)\in \cD(t)$. 
For $\cG\subseteq \I$, let $\mho: \mathbb{R}\to [0, 1]$ be a smooth weight which is $(\cG, \kappa)$ adapted for some small enough $\kappa$. Recall the constant $L$ from \eqref{def L and Crho}. For $1\leq j\leq n$, we define
\begin{equation}
\label{def Nj}
N_{j,\mho}(\bdel, t)=
\sum_{(q,a)\in \Z^{2}} 
({\mho}\circ f_j^{-1})\left(\frac{a}{q}\right) 
\ome\left(\frac{q}{2^t}\right)
\prod_{\substack{1\leq i\leq n\\ i\neq j}} 
w\bigg(
\frac{\|q (f_i\circ f_j^{-1}) (\frac{a}{q}) \|}{4L\delta_i}
\bigg).
\end{equation}
The following lemma details how $\cR(\cG, 4\bdel, t)$ can be dominated by the smooth count $N_{j, \mho}(\cG, \bdel, t)$, where the index $j\leq n$ is such that
$\min_{1\leq i\leq n} \delta_i = \delta_j$.

\begin{lemma}
\label{lem rough to smth}
Let $\rho$ and $L$ be as in \eqref{eq rho assump} and \eqref{def L and Crho} respectively. 
Let $1\leq j\leq n$ be such that $\min_{1\leq i\leq n} \delta_i=\delta_j$, and let $\mho$ be $(\cG, 4L\delta_j/2^t)$ adapted.
Then, for any $\cG \subseteq \I$, $t\in \mathbb{N}$ and $\bdel=(\delta_1, \ldots, \delta_n)\in \cD(t)$,
\begin{equation}
\label{eq Rt Nj comp}
\#\cR(\cG, 4 \bdel, t) \leq N_{j, \mho}(\cG, \bdel, t).
\end{equation}
\end{lemma}

\begin{proof}
Define
\begin{equation*}
\label{eq def Sj set}
\cS_j(\cG, \bdel, t) =
\Big\{(a, q)\in \mathbb{Z}^2: 
2^t\leq q\leq 2^{t+1} \text{ and }
\eqref{eq something} \text{ holds} \Big\},
\end{equation*}
where $\eqref{eq something}$
refers to the system of inequalities
\begin{equation}
\label{eq something}
\min_{\dx\in \cG}\Big\vert 
\dx-f_j^{-1}\left(\frac{a}{q}\right)\Big\vert
< \frac{4L\delta_j}{q}, \quad 
\left\|qf_i\circ f_j^{-1}\left(\frac{a}{q}\right)\right\|<4L\delta_i \textrm{ for } 1\leq i\neq j\leq n.
\end{equation}
Using the support properties of the smooth cutoff functions defined above, it is straightforward to check that 
\begin{equation}
\label{eq Nj lb ct}
\#\cS_j(\cG, \bdel, t)\leq N_{j, \,\mho}(\cG, \bdel, t).
\end{equation}
Let $\pi_j: \cR(\cG, 4\bdel, t)\to \mathbb{Z}^2$ be the projection map
$$\pi_j(\bfa, q)=(a_j, q).$$
We claim that $\pi_j$ maps $\cR(\cG, 4\bdel, t)$
injectively to $\cS_j(\cG, \bdel, t)$.
This claim, combined with \eqref{eq Nj lb ct}, will imply the desired estimate \eqref{eq Rt Nj comp}.

\bigskip

To prove the claim, let $(\bfa, q)\in \cR(\cG, 4\bdel, t)$. Then
$2^t\leq q< 2^{t+1}$, and there exists $\dx_0\in \cG$ such that
\begin{equation}
\label{eq xi exists}
\Big\vert f_i(\dx_0)-\frac{a_i}{q}\Big\vert
<\frac{4\delta_i}{q}
\quad \textrm{ for all } 1\leq i\leq n.
\end{equation}
Let $C:=C_\rho$ be as in \eqref{def L and Crho}.
For $1\leq i\leq n$ with $i\neq j$, using the triangle inequality, the mean value theorem for $f_i\circ f_j^{-1}$, and \eqref{eq xi exists}, we get
\begin{align*}
\left|f_i\circ f_j^{-1}\left(\frac{a_j}{q}\right)-\frac{a_i}{q}\right|& \leq \left|f_i\circ f_j^{-1}\left(\frac{a_j}{q}\right)-f_i(\dx_0)\right|+\left|f_i(\dx_0)-\frac{a_i}{q}\right|\\
& \leq C \left|\frac{a_j}{q}-f_j(\dx_0)\right|+\frac{4\delta_i}{q}\\
&\leq  \frac{4C\delta_j}{q}+\frac{4\delta_i}{q}<\frac{4L\delta_i}{q},
\end{align*}
where we have used $\delta_j\leq \delta_i$ to get the last inequality. 
Arguing similarly and using the mean value theorem for $f_j^{-1}$, we can also see that
$$\left|x_0 -f_j^{-1}\left(\frac{a_j}{q}\right)\right|<\frac{4L\delta_j}{q}.$$
This shows that $\pi_j$ maps $\cR(\cG, 4\bdel, t)$ to $\cS_j(\cG, \bdel, t)$. 

To see that this mapping is injective, let $(\bfa, q_1), (\bfb, q_2)\in \cR(\cG, 4\bdel, t)$ with $\pi_j((\bfa, q_1))=\pi_j( (\bfb, q_2))$. This immediately gives $a_j=b_j$ and $q_1=q_2=q$ (say). Further, let $\dx_1, \dx_2\in \cG$ be such that
\begin{equation}
    \label{eq exists xi1 xi2}
    \left|f_i(\dx_1)-\frac{a_i}{q}\right|< \frac{4\delta_i}{q}, \qquad \left|f_i(\dx_2)-\frac{b_i}{q}\right|< \frac{4\delta_i}{q}, \qquad \textrm{ for } 1\leq i\leq n. 
\end{equation}
Using the triangle inequality and the above for $i=j$, we get
$$\left|f_j(\dx_1)-f_j(\dx_2)\right|< \frac{8\delta_j}{q}.$$
For each $1\leq i\leq n$ with $i\neq j$, we again use the triangle inequality, the mean value theorem for $f_i \circ f_j^{-1}$ and \eqref{eq exists xi1 xi2} to estimate
\begin{align*}
\left|\frac{a_i}{q}-\frac{b_i}{q}\right|
& \leq \left|f_i(\dx_1)-\frac{a_i}{q}\right|
+\left|f_i(\dx_1)-f_i(\dx_2)\right|
+\left|f_i(\dx_2)-\frac{b_i}{q}\right| \\
& \leq \frac{8\delta_i}{q}+\frac{8C\delta_j}{q}
\le \frac{4L\delta_i}{q}.
\end{align*}
Because $\Xi$ is large, see \eqref{def Xi}, 
we are ensured $L\delta_i\leq L 2^{-10} L^{-1}<1/4$
and hence $a_i=b_i$ for all $1\leq i\leq j$. This shows that $\pi$ is injective, thus
completing the proof of the claim and also the lemma. 
\end{proof}

\subsection{Final reduction}
\label{sec final reduc}
In view of Lemma \ref{lem rough to smth}, to establish Theorem \ref{thm: prob count}, it suffices to show that the following holds for all $t \in \bN$ and all $\bdel\in \cD(t)$. 
Let $\delta_{\textrm{min}}:=\min_{1\leq i\leq n} \delta_i=\delta_j$.
Then there exists a set $\cA_\bdel\subset \I$ satisfying
\eqref{eq: convergence of measures} and
a function $\Omega_{\textrm{g}}=\Omega_{\textrm{g}, \bdel, t}$ 
which is $\left(\I\setminus \cA_{\bdel}, 4L\delta_{\textrm{min}}/2^t\right)$ adapted such that
\begin{equation}
    \label{eq smooth Bdel count}
   N_{j, \Omega_{\textrm{g}}}(\bdel, t) \ll_{\bff, \rho} 2^{2t}\bdel^{\times},
\end{equation}
where the implicit constant is independent of $t$ and $\bdel$.

We shall show this in the case when $j=1$ and $\delta_1=\min_{1\leq i\leq n} \delta_i$. The other cases follow by a relabelling of coordinates. For $2\leq i\leq n$, we define $F_i: f_1(U)\to \mathbb{R}$ to be
\begin{equation*}
    \label{def repar f}
    F_i=f_i\circ f_1^{-1},\qquad \bfF(\dx)=(F_2(\dx), \ldots, F_n(\dx)).
\end{equation*}
Note that if $\bff$ is smooth and $l$-non-degenerate on $U$, then so is $\bfF$. Further, its derivatives satisfy the estimate
\begin{equation*}
    \label{eq repar F der est}
    \sup_{y\in U}\Vert \bfF^{(N)}(y)\Vert_\infty \ll_N \rho^{-N}
\end{equation*}
for each $N\geq 1$. Thus, replacing $\bff(x)$ by $(x,\bfF(x))$, and renaming $f_1(U)$ and $f_1(\I)$ to $U$ and $\I$ respectively, we may assume without loss of generality that the curve $\cC$ is in the parametrised form
$$\cC= \{(\dx, \mathbf{F}(\dx)): \dx\in \I\subset U\},
$$
where 
$
\mathbf{F}=(\bfF_2(\dx),\ldots, \bfF_n(\dx)): U\to \mathbb{R}^{n-1}
$
is smooth and non-degenerate.

\subsection{Quantitative non-divergence and sub-level set decomposition}
\label{sub-levelFA}

\subsubsection{Quantitative non-divergence in the space of lattices}
\label{subsec quant nondiv}
To move forward, our method needs deep input from homogeneous dynamics.
It gives a useful estimate for the measure of this set under the assumption that $\bfF$ 
is non-degenerate (recall Definition \ref{def nondeg}).
The following is a special case of \cite[Theorem 1.4]{BKM2001}.

\begin{theorem}[Quantitative non-divergence estimate]
\label{thm quant non div}
Let $U_0\subseteq \R$ be an open set. 
Let $\bG=(G_1\ldots,G_n): U_0 \rightarrow \R^n$
be $l$-non-degenerate. Given
$\mathbf{V} = (V_1, \ldots, V_n)\in \R_{\geq 1}^{n}$, 
$\Delta\in (0, 1]$, and $K>0$,
consider the sub-level set
\begin{equation}
\label{def sub-level set}
\mathfrak{S}_{\Delta, K, \bV}^{\bG, U_0} = \left\{
\dx\in U_0: 
    \exists \bfv\in \Z^{n}\setminus\{\bzero\}
    \quad
    \begin{array}{l}
   \Vert \langle \bG(x),\bfv \rangle \Vert 
   < \Delta, \\
   |\langle \bG'(x),\bfv \rangle| < K, \\
   |v_1| < V_1,\ldots, |v_n| < V_n  
  \end{array}
    \right\}.
\end{equation}
Suppose that
\begin{equation}
    \label{Cond 1}
    \tag{Cond 1}
   \Delta K \frac{V_1 V_2\ldots V_n}
   {\max_{1\leq i\leq n}V_i}<1.
\end{equation}
Then there exists a constant $C=C(\bG)>0$ such that
\begin{equation*}
\label{eq nondiv glob}
\lam_1(\mathfrak{S}_{\Delta, K, \bV}^{\bG, U_0}) 
\leq C \cdot  
\left( \Delta K \frac{V_1V_2\ldots V_n}{\max_{1\leq i\leq n}V_i}\right)^{\frac{1}{(2l-1)(n+1)}}.
\end{equation*}
\end{theorem}

\subsubsection{Sub-level set decomposition}
\label{subsec sub-level}
Roughly speaking,
the objective of this subsection is to 
partition $\I$ smoothly 
into a `good' and a `bad' region. 
To make this more precise, 
we introduce more notation.
Throughout this section, 
we fix $t\in \mathbb{N}$ and $\bdel\in \cD(t)$. 
Further, we fix another small parameter $\eta\in (0, 1/2)$ to be chosen at the end depending on $\theta$ from \eqref{eq lower bound psi} but independent of $\bdel$ and $t$.
Let 
\begin{equation}
\label{eq def Q Ti}
Q=Q_t=2^t, \qquad T_i=\frac{2C_\bfF Q^{\eta}}{\delta_i} \textrm{ for } 1\leq i\leq n.
\end{equation}
Here $C_{\bfF}>1$ is a large constant 
depending on $\bfF$ and $\rho$, see \eqref{eq grad lb}.
All constants, implicit or explicit, are allowed to depend on $\bff$, $\rho$ and $\eta$, and hence also on $\theta$.
The smallest $\delta_i$
plays a special role. As we argued earlier,
we may assume --- after possibly relabelling
and re-parametrising --- that 
\begin{equation}
\label{def del min T max}
\delta_{\textrm{min}} := \min_{1\leq i\leq n}\delta_i=\delta_1 \quad
\mathrm{and} \quad 
T_{\textrm{max}}:= \max_{1\leq i\leq n} T_i= \frac{2C_\bfF Q^{\eta}}{\delta_1}.
\end{equation}
Throughout
the present section, we shall abbreviate
\begin{equation}
\label{def sub-level set some parameter less}
\mathfrak{S}_{\Delta, K}:=
\mathfrak{S}_{\Delta, K, \mathbf{T}}^{\bf, U},   
\end{equation}
where the parameters $\Delta, K$ are
still for us to choose (in terms of $\bdel$ and $t$).

We construct two smooth functions 
with the property that their sum 
majorises the indicator function $1_{\I}$.
The first of these lives on a small $r$ 
neighbourhood of the sub-level set 
$\mathfrak{S}_{\Delta, K}$ 
given by \eqref{def sub-level set}, 
for a suitably chosen value of $r$ depending on $\bdel$ and $t$.
Let 
\begin{equation*}
\label{def Phi}    
\Phi_{\bfv}(\dx) = 
\langle (\dx, \bfF(\dx)), \bfv\rangle=xv_1+\langle \bfF(\dx), \widetilde{\bfv}\rangle, \textrm{ where } \bfv=(v_1, \widetilde{\bfv}).
\end{equation*}
With the above notation, we can write 
\begin{equation*}
\label{eq sub-level set specialised}
\mathfrak{S}_{\Delta, K} 
    =
    \left\{
    \dx\in U: 
    \exists \bfv\in \Z^{n}\setminus\{\bzero\}
\quad
    \begin{array}{l}
   \Vert \langle \Phi_{\bfv}(x),\bfv \rangle \Vert 
   < \Delta, \\
   |\langle \Phi_{\bfv}'(x),\bfv \rangle|
   < K, \\
   |v_1| < T_1,\ldots, |v_n| < T_n  
  \end{array}
    \right\}.
\end{equation*}
The measure of this set can be estimated 
by Theorem \ref{thm quant non div}. 

The other smooth function is supported on 
the $r$-neighbourhood of the complement 
$\I\setminus \mathfrak{S}_{\Delta, K}$ 
and is the region where Fourier-analytic 
techniques can be employed to count 
rational points in the vicinity of $\cC$. 
To this effect, we construct a smooth cutoff function supported on an $r$-neighbourhood of the set $\mathfrak{S}_{\Delta, K}$, with $\Delta$ and $K$ to be chosen at the end in \S \ref{subsec par choice}.

\begin{lemma}[Smooth cutoff]
\label{le smooth cut off}
For $\Delta\in (0, 1/2]$ and $K>0$, let $\mathfrak{S}_{\Delta, K}$ be as defined in \eqref{def sub-level set}.
Suppose that 
\begin{equation}
    \tag{Cond 2}
    \label{Cond 2}
    0<r\leq  \frac{Q^{-2\eta}}{2}\min(K\delta_{\textnormal{min}},\Delta/K).
\end{equation}
Then there exists a smooth function 
$W_{r} = W_{r,\mathfrak{S}_{\Delta, K}}:U \rightarrow [0,1]$
with the following properties. 
\begin{enumerate}
    \item The function $W_r$ is identically $1$ on $\mathfrak{S}_{\Delta, K}$. 
    
    \item The function $W_r$ is $r$-rough, meaning 
    its derivatives satisfy
    \begin{equation}
        \label{eq r rough est}
        \left\|W_r^{(N)}\right\|_{\infty}
         \ll_N r^{-N}
    \end{equation}
    for any non-negative integer $N$. 
    \item The support of $W_r$ can be covered by a union of $T_r$ many open balls of radius $r$ 
         with $T_r=O\left(\frac{\lam_1(\fS_{2\Delta, 2K})}{r}\right)$.
         More precisely, 
         \begin{equation}
             \label{eq supp wr}
             \supp\, W_r\subseteq \bigcup _{i\leq T_r} (x_i-r, x_i+r),
         \end{equation}
        with $x_i\in U$ for $1\leq i\leq T_r$.
         Moreover, 
         \begin{equation}
             \label{eq sup supp wr}
             \bigcup _{i\leq T_r} (x_i-2r, x_i+2r)\subseteq \fS_{2\Delta, 2K}.
         \end{equation}
\end{enumerate}
\end{lemma}

\begin{proof}
The statement follows from \cite[Lemma 4.1]{SST}, where 
$\delta$ is now replaced by $\delta_{\mathrm{min}}$ and $r$ is replaced by $r/2$.
\end{proof}

Let $w$ and $\omega$ be as defined in \S \ref{subsec smooth cutoff} and fixed from now. We also fix a smooth function ${\Omega}:\mathbb{R}\to [0, 1]$ with support contained in $U$ which is identically $1$ on $\I$.  
The smooth cutoff $W_r$ from Lemma \ref{le smooth cut off} allows us to decompose the curve $\cC$ into a `bad' part corresponding to an $r$-neighbourhood of $\fS_{\Delta, K}$ and a `good' part amenable to counting via Fourier analysis.
For $\Delta\in (0, 1/2]$ and $K>0$, let $\mathfrak{S}_{\Delta, K}$ be as defined in \eqref{def sub-level set}. Moreover, for $r$ satisfying \eqref{Cond 2}, let $W_r$ be as in Lemma \ref{le smooth cut off}. Define \begin{align*}
\wbad: U &\to [0, 1] \\
x &\mapsto \Omega(\bfx) W_r(\bfx).
\end{align*}
We say that $\wbad$ is the sub-level 
part of $\Omega$ with respect to $(F, \Delta, K, r)$. 
Further, let $\wgood: U\to [0, 1]$ be the good part of $\Omega$ with respect to $(F, \Delta, K, r)$, defined to be
\begin{equation}
    \label{def wgood}
    \wgood =\Omega  - \wbad =\Omega \cdot \left (1- W_r\right ).
\end{equation}

\subsection{Fourier analysis on the good part}
\label{subsec fourier expansion}

The aim of this section is to establish an upper bound for the number of rational points with denominators of size about $2^t$ in an anisotropic $\bdel$ neighbourhood of the portion of $\cC$ parametrised as a graph over the complement of the set $\fS_{\Delta, K}$. Given a natural number $t$ and $\bdel\in \cD(t)$, we define the counting function
\begin{equation}
    \label{def N repar}
    N_{\Omega_{\Delta, K, \textrm{g}}}(\bdel, t) =
\sum_{(q,a)\in \Z^{2}} 
{\Omega_{\Delta, K, \mathrm{g}}}\left(\frac{a}{q}\right) 
\ome\left(\frac{q}{2^t}\right)
\prod_{\substack{2\leq i\leq n}} 
w\bigg(
\frac{\|q F_i (\frac{a}{q}) \|}{\delta_i}
\bigg).
\end{equation}
Here $\Delta=\Delta(\bdel, t)\in (0, 1/2]$ and $K=K(\bdel, t)>0$ are parameters to be chosen at the end.

Recall \eqref{eq def Q Ti}, \eqref{def del min T max} and 
$$\bdel^{\times}=\frac{\prod_{1\leq i\leq n}\delta_i}{\delta_{\textrm{min}}}=\prod_{2\leq i\leq n} \delta_i.$$
For $1\leq i\leq n$, put
\begin{equation}
    \label{def J}
    J_i = Q^{\eta}\delta_i^{-1}, \qquad J_{\max} =J_1.
\end{equation}
In the Fourier expansion of $N_{\Omega_{\Delta, K, \textrm{g}}}(\bdel, t)$,
we shall encounter frequency vectors in the set
\begin{equation*}
\label{def J'}
    \cJ':= 
    \{(j_2,\ldots,j_n)\in \Z^{n-1}\setminus\{\bzero\}:
    0\leq \vert j_i\vert \leq J_i
    \textrm{ for all } 2\leq i \leq n\}.
    \end{equation*}
To each $\bfj'\in \Z^{n-1}$, the coefficients
\begin{equation}
    \label{def B coef}
    B(\bfj'):=B_{\bdel}(\bfj'):= \prod_{2\leq i\leq n}
\widehat{\bump}(\delta_i j_i)
\end{equation}
will effectively truncate the Fourier analysis to $\cJ'$.

\subsubsection{Extracting the main term}
A simple Fourier expansion argument gives rise to the expected main term contribution to $N_{\wgood}(\bdel, t)$.

\begin{lemma}
\label{le decomp}
We define the proposed
main term 
\begin{equation}
\label{def main term}
 M(\bdel, Q):=   
 \bdel^{\times} \widehat{\bump}(0)^{n-1}
\sum_{(q,a)\in \Z^{2}} 
\wgood\left(\frac{a}{q}\right) 
\ome\left(\frac{q}{Q}\right)\ll \bdel^{\times} Q^2
\end{equation}
stemming from the following truncated Fourier expansion of 
$N^{\wgood}(\bdel, Q)$.
The function 
$N^{\wgood}(\bdel, Q)$ equals $M(\bdel, Q)$
plus the smooth exponential sum
\begin{equation*}
\label{def Error Term}
E(\bdel, Q):=
\bdel^{\times}   
\sum_{(q,a,\bfj')\in \Z^{2}\times \cJ'} 
\wgood\left(\frac{a}{q}\right) 
\ome\left(\frac{q}{Q}\right)
B(\bfj') e(q \langle 
\bfF(a/q), \bfj' \rangle),
\end{equation*}  
where $B(\bfj')$ is as defined in \eqref{def B coef}, 
up to a rapidly-decaying error in $Q$. Formally,
\begin{equation} 
\label{eq Fourier decomp}
N^{\wgood}(\bdel, Q)
=
M(\bdel, Q)
+ E(\bdel, Q)
+ O_{w,\eta,A}(Q^{-A})
\end{equation}
for any $A>1$.
\end{lemma}

\begin{proof}
We use the Fourier expansion of
$x\mapsto w(\|x\|/\delta_i)$ 
for $2\leq i\leq n$, see
\cite[Lemma A.3]{ST}.
This gives
\begin{align*}
&\bump \left(
\frac{\|q F_2 (a/q) \|}{\delta_2} \right) \cdots 
\bump \left(
\frac{\|q F_n (a/q) \|}{\delta_n} \right) \\
&=
\prod_{2\leq i\leq n} \Big(
\delta_i \widehat{\bump}(0)
+ \delta_i  
\sum_{1 \leq \vert j_i \vert 
\leq J} 
\widehat{\bump}(\delta_i j_i) 
e(j_iq F_i (a/q ))
+ O_{\bump,\eta,A}(Q^{-A})
\Big) \\
&= {\bdel}^{\times} \widehat{\bump}(0)^{n-1}
+ {\bdel}^{\times} \sum_{\bfj' \in \cJ'} 
B(\bfj')
e\left(q \langle 
\bfF( a/ q ),
\bfj' \rangle\right)
+ O_{\bump,\eta,A}(Q^{-A}).
\end{align*}
Substituting this into the 
definition of $N^{\wgood}$ yields
\eqref{eq Fourier decomp}.

We still need to verify the upper bound in 
\eqref{def main term}.
To this end, note that 
$\wgood  \ll  1_{[-4, 4]}$ 
and $\omega  \ll  1_{[-4, 4]}$,
where $1_{[-4, 4]}$ is the indicator function 
of $[-4,4]$.
Replacing 
$\wgood$ and $\omega$ by $O(1_{[-4, 4]})$,
and using that there are $O(Q^2)$ many
fractions $a/q$ with $q\ll Q$ in $[-4, 4]$,
produces the required bound.
\end{proof}

\subsubsection{Pruning the error term} 
\label{subsec ibp}
It remains to show that
$E(\bdel, Q)$
is indeed an error term.
To this end, we write it as an average of oscillatory integrals.
To estimate the oscillatory integrals, we shall need to keep track of the terms arising from repeated 
integration by parts, keeping in mind the amplitude blowup. Recall that our weight function $\wgood$ is $r$-rough,
see \eqref{eq r rough est}.

The basic setup is as follows. Let $h\in C_c^\infty (\R)$ and $\psi\in C^\infty(\R)$ be real-valued functions such that 
$\psi'\neq 0$ on the support of $h$. 
We define the differential operator $L$
via
\begin{equation}
 \label{def Lop} 
 L h
= - \: \frac{\mathrm d}{\mathrm d\dx} \left(\frac{h}{\psi'}\right).
\end{equation}
Its formal adjoint operator is given by  
$ L^*= -\frac{1}{\psi'}\frac{\mathrm{d}}{\mathrm{d} \dx}$.
Note that it has the property
$(2\pi \mathrm{i} \lambda) ^{-1}L^* e(\lambda \psi)
=e(\lambda \psi)$.
Thus, we can express an $N$-fold integration by parts as
\begin{equation} \label{Nfold}
\int_{\R} e(\lambda \psi(x)) 
h(x) \rd x = 
(2\pi \mathrm{i}\lambda )^{-N} 
\int_{\R} e(\lambda \psi(x)) 
(L^{\circ N} h)(x) \rd x
\end{equation}
for $N\geq 1$.
We can express $L^{\circ N}(h)$ efficiently using the following lemma.

\begin{lemma}[Lemma A.2 in \cite{ACPS} with $n=1$]
\label{lem partial}
Let $N$ be a positive integer, $h\in C_c^N (\R)$ and $\psi\in C^{N+1}(\R)$. Suppose that 
$\psi'\neq 0$ on the support of $h$. Then 
there exists an integer $V=V(N)\geq 1$ and complex coefficients $c_{N, 1}, \ldots, c_{N, V}$ so that
$$
L^{\circ N} h = \sum_{\nu=1}^{V} c_{N,\nu} h_{N,\nu},$$ 
where
$$
h_{N,\nu} = \frac{h^{(\kappa_0)}}{|\psi'|^{\kappa_0}} \prod_{l=1}^M \frac{\psi^{(\kappa_l+1)}}{|\psi'|^{\kappa_l+1}},
$$
for some non-negative integer $\kappa_0$ 
and strictly positive integers 
$\kappa_1, \ldots, \kappa_M$ satisfying
\begin{equation}
\label{eq order}
\kappa_0+\sum_{l=1}^M \kappa_l=N.
\end{equation} 
\end{lemma}

Since $\kappa_l>0$ for each $l\leq M$, Equation \eqref{eq order} implies that
\begin{equation}
    \label{eq jM sum}
    \kappa_0+M\leq N.
\end{equation}
We now turn our attention to the analysis of 
$E(\bdel, Q)$. Let 
\begin{equation*}
\label{eq v split}
\bfj = (j_1, \bfj')\in \mathbb{R}\times\mathbb{R}^{n-1}.
\end{equation*}
For a large enough constant $C_\bfF>0$, depending only on $\bfF$ and to be decided shortly, we set
\begin{equation}
\label{def Vset}
\cJ:=\bigg 
\{(j_1, \bfj')\in \Z^n: 
\begin{array}{ll}
& 0 \leq |j_1| \leq C_\bfF J_{\textrm{max}}, \quad
\bfj'\neq \bzero,  \\
& |j_2|\leq J_2,\ldots,|j_n|\leq J_n    
\end{array} 
\bigg\}.   
\end{equation}
We define the truncated error term 
$E_{\mathrm{tr}}(\bdel, Q)$ to be 
\begin{equation*}
    \label{def trunc error} 
{\bdel}^{\times}   Q^{2}
    \sum_{
    \substack{0 \leq \vert c\vert 
    \leq C_\bfF J_{\textrm{max}}
    \\ \bfj\in \cJ}}
    B(\bfj')
    \int_{\R^{2}} 
    e\left(y Q 
    [\Phi_{\bfj}(\bfx) - c] \right)
    \wgood(\bfx)\,y\cdot 
    \ome(y)\,
    \rd \bfx\, \rd y,
\end{equation*}
where
\begin{equation*}
\label{def Phij}    
\Phi_{\bfj}(\dx) = 
\langle (\dx, \bfF(\dx)), \bfj\rangle=xj_1+\langle \bfF(\dx), \bfj'\rangle.
\end{equation*}
The next lemma shows that up to an acceptable error, it is  $E_{\mathrm{tr}}(\bdel, Q)$ which decides the order of magnitude of $E(\bdel, Q)$.
\begin{lemma}
\label{lem pruning error} 
Let  $N\geq 4$.
Assume that 
\begin{equation}
\label{eq Cond 2+}
\tag{Cond 3}
r\geq Q^{\eta-1}\delta_{\textnormal{min}}.
\end{equation}
Then
\begin{equation*}
\label{eq E truncated}
E(\bdel, Q) =
E_{\mathrm{tr}}(\bdel, Q)
    + O_N({\bdel}^{\times} J_{\textnormal{max}}^{n+1} 
    Q^{2 } (QrJ_{\textnormal{max}})^{-N} 
+{\bdel}^{\times}Q^{-2N}).   
\end{equation*}
\end{lemma}

\begin{proof}
The proof is similar to \cite[Lemma 4.5]{SST}, 
except here we operate only under the weaker condition \eqref{eq Cond 2+}.
This turns out not to be an issue. The shape of the error term is also slightly different. We give the details in Appendix \ref{sec proof lemma pruning} for completeness.    
\end{proof}

Suppose 
\begin{equation}
    \label{eq Cond 2++}
    \tag{Cond 4}
    \delta_{\textrm{min}} \geq 2 Q^{-\eta -1}.
\end{equation}
Observe that, when this and \eqref{eq Cond 2+} are satisfied, we have 
\begin{equation}
\label{Jbounds}
J_{\textrm{max}}\leq Q^{1+2\eta}, \qquad QrJ_{\textrm{max}}\geq Q^{2\eta}.
\end{equation}
Thus, by Lemma \ref{lem pruning error} and choosing $N>\frac{n+2}{\eta}$ with $\eta<\frac{1}{2}$, 
\begin{equation}
    \label{eq le ee small concl}
    E(\bdel, Q)=E_{\mathrm{tr}}(\bdel, Q)
    + O_{n, \eta}({\bdel}^{\times}Q).
\end{equation}

\subsubsection{Bounding the error term}
The final lemma of this subsection establishes that $E(\bdel, Q)$ is indeed an acceptable error term.
\begin{lemma}
\label{le error small} 
Assume \eqref{eq Cond 2+} and \eqref{eq Cond 2++} are true.
Suppose that
\begin{equation}
     \tag{Cond 5}
    \label{Cond 5}
    \Delta\geq Q^{\eta-1}
\end{equation}
and
\begin{equation}
    \tag{Cond 6}
    \label{Cond 6}
    Q^{\eta-1} K^{-1} \leq r \leq Q^{-\eta} K\delta_{\textnormal{min}}\leq 1.
\end{equation}
Then 
$
E(\bdel, Q) \ll_\eta  {\bdel}^{\times} Q.
$
\end{lemma}

\begin{proof}
We first bound $E_{\mathrm{tr}}(\bdel, Q)$.
Fix $\bfj \in \cJ$ as defined in \eqref{def Vset}.
Notice that for any point $x \in \supp(\wgood)$,
at least one of the inequalities
\begin{equation*}
\label{eq inequalities to be reversed}
\Vert \Phi_{\bfj}(x)\Vert <  \Delta,
\qquad
\vert \Phi_{\bfj}'(x) \vert < K
\end{equation*}
is false.
This motivates us to decompose 
$E_{\mathrm{tr}}(\bdel, Q)$ based on the size of 
$\vert \Phi_{\bfj}'(x) \vert$ with respect to $K$. 
Let $b:\mathbb{R}\to [0,1]$ be a smooth weight, 
supported on $(-1, 1)$ and 
identically equal to $1$ on 
$[-1/2, 1/2]$. 
For each $\bfj=(j_1, \bfj')\in \cJ$ and 
$c\in \mathbb{Z}$ with $0 \leq |c| \leq C_{\mathbf{f}}J_{\textrm{max}}$, we split
\[\int_{\R^{2}}
    y e(y Q 
    [\Phi_{\bfj}(x) - c] )
    \wgood(\bfx) 
    \ome(y)
    \rd \bfx \rd y
    =  I(c, \bfj)+ II(c, \bfj),
\]
where
\begin{align*}
I(c, \bfj) &= \int_{\R^{2}}
    y e(y Q 
    [\Phi_{\bfj}(x) - c] )
    \wgood(x) 
    \ome(y) b(K^{-1}\Phi_{\bfj}' (x))
    \rd x \rd y, \\
II(c, \bfj) &= \int_{\R^{2}}
    y e(y Q 
    [\Phi_{\bfj}(x) - c] )
    \wgood(x) 
    \ome(y) (1-b(K^{-1}\Phi_{\bfj}' (x)))
    \rd x \rd y.
\end{align*}
    
We deal with the second integral first. 
The amplitude function in $II(c, \bfj)$ 
is supported on the set where 
$\vert \Phi_{\bfj}' (\bfx)\vert 
\geq K/2$, which lets us integrate by parts in the $\bfx$ variable, keeping $y$ fixed.  
Let $L$ be the differential operator defined in \eqref{def Lop} with $\psi=\Phi_{\bfj}$. Set $$
h=\wgood\cdot [1-b(K^{-1}\Phi_{\bfj}')].$$ Using Lemma \ref{lem partial}, we get that $L^{\circ N} h = \sum_{\nu=1}^{V} c_{N,\nu} h_{N,\nu}$
with each $h_{N,\nu}$ of the form $$\frac{h^{(\kappa_0)}}{|\psi'|^{\kappa_0}} \prod_{l=1}^M \frac{\psi^{(\kappa_l+1)}}{|\psi'|^{\kappa_l+1}}\quad \textrm{ with } \kappa_0+\sum_{l=1}^M \kappa_l=N.$$
Let $\kappa\in \mathbb{Z}_{\geq 0}$. By the Leibniz rule, we have
\begin{equation}
    \label{eq amp blowup}
    \| h^{(\kappa)}\|_{\infty} \ll (K^{-1}Q^{\eta}\delta_{\textrm{min}}^{-1}+r^{-1})^{\kappa}\ll r^{-\kappa},
\end{equation}
where we used $r\leq \delta_{\textrm{min}} K Q^{-2\eta}$ to get the last estimate.
Combining this with the assumed lower bound on 
$\vert  \Phi_\bfj \vert $, we get 
$\left\|\frac{h^{(\kappa_0)}}{(\psi')^{\kappa_0}}\right\|_{\infty}\ll (rK)^{-\kappa_0}.$

Since
$\|\psi^{(\kappa+1)}\|_{\infty}\ll Q^{\eta}\delta_{\textrm{min}}^{-1}$, and using the lower bound on 
$ \vert \psi'\vert =
\vert \Phi_\bfj' \vert $, we get
$$
\left\|\frac{\psi^{(\kappa_l+1)}}{(\psi')^{\kappa_l+1}}\right\|_{\infty}\ll 
Q^{\eta} \delta_{\textrm{min}}^{-1}K^{-(\kappa_l+1)}.
$$
We combine \eqref{eq order}, \eqref{eq jM sum} and the above observations to conclude
\begin{align*}
\|h_{N, \nu}\|_\infty\ll (rK)^{-\kappa_0}\prod_{l=1}^M (Q^{\eta}\delta_{\textrm{min}}^{-1} K^{-(\kappa_l+1)})&=r^{-\kappa_0}Q^{M\eta}K^{-N}(\delta_{\textrm{min}} K)^{-M}\\
&\leq r^{-\kappa_0} K^{-N} (Q^{-\eta}\delta_{\textrm{min}} K)^{-N+\kappa_0}. 
\end{align*}
Recalling that $1\leq (Q^{-\eta}\delta_{\textrm{min}} K)^{-1}\le 
r^{-1}$, and noting that $(rK)^{-1} \leq Q^{1-\eta}$ 
by condition \eqref{Cond 6}, we obtain
$$\|h_{N, \nu}\|_\infty\ll  r^{-\kappa_0}K^{-N}r^{-N+\kappa_0}\ll Q^{(1-\eta)N}.$$
Thus $Q^{-N}\|L^{\circ N} h\|_{\infty}\ll_{N, d} Q^{-N\eta},$
and integrating by parts $N$ many times with respect to $\bfx$ using \eqref{Nfold}, yields
$$ II(c,\bfj) \ll Q^{-N\eta}.$$

\bigskip

We now bound $I(c,\bfj)$. Observe that the support of $ \wgood\cdot b(K^{-1}\Phi_{\bfj}' )$ is contained in the intersection of the complement of $\mathfrak{S}_{\Delta, K}$
with $\{ \|\Phi_{\bfj}'\|_\infty <K \}$. 
By the definition of $\mathfrak{S}_{\Delta, K}$, we conclude that for any $c\in\mathbb{Z}$, $$|\Phi_{\bfj}(\bfx)-c| \geq \|\Phi_{\bfj}(\bfx) \|\geq \Delta\geq Q^{\eta-1}$$ on this set. Thus, we integrate by parts in the $y$ variable with $$\psi(y)=y 
[ \Phi_{\bfj}(x) - c]).
$$ Since 
 $|\omega^{(\kappa)}(y)|\ll_\kappa 1$ and $\psi^{(\kappa+1)}(y) =0 $ for any $\kappa\in\mathbb{Z}_{\geq 1}$, we conclude that 
 $$I(c,\bfj)\ll_{N} Q^{-N \eta}$$
 for any $N\in \N$.
Summing over $\bfj\in \mathcal{J}$ 
and $0<|c|<C_{\bfF}J$ gives
$$E_{\mathrm{tr}}(\bdel, Q)\ll_N {\bdel}^{\times} 
Q^{2}
    \sum_{0 \leq \vert c\vert \leq 
    C_\bfF J_{\mathrm{max}}} 
    \sum_{\bfj\in \cJ} Q^{-N\eta}
    \ll_N {\bdel}^{\times} Q^{2} J_{\textnormal{max}}^{n+1}Q^{-N\eta}.$$
Recalling \eqref{Jbounds} 
and choosing $N>\frac{2(n+2)}{\eta}$ delivers
$
E_{\mathrm{tr}}(\bdel, Q)
\ll_\eta \bdel^{\times} Q.
$
In view of \eqref{eq le ee small concl}, we 
finally have $E(\bdel, Q)\ll_\eta {\bdel}^{\times} Q$.
\end{proof}

\subsection{Proof of Theorem \ref{thm: prob count}}
\label{subsec proof counting}
\subsubsection{Choice of parameters}
\label{subsec par choice}
Let $Q=2^t$ and $\bdel=(\delta_1, \ldots, \delta_n)\in \cD(t)$. For a small enough $\eta>0$, independent of $t$ and $\bdel$, we now set
\begin{equation}
\label{eq par choice}
T_i=2C_\bfF Q^{\eta} \delta_i^{-1}=2C_{\bfF} J_{i}, 
\,\,\Delta=Q^{10\eta-1},
\,\, K= Q^{-4\nu} \delta_{\textrm{min}}^{-1},
\,\,
r=\frac{Q^{4\eta-1}}{K}.
\end{equation}
With the above choice of parameters:
\begin{enumerate}[(i)]
\item We can estimate
\begin{align}
\label{eq measure comp}
\Delta K \frac{T_1T_2\cdots T_n}{\max_{1\leq i\leq n}T_i}
&\ll Q^{10\eta-1} Q^{-4\nu}\delta_{\textrm{min}}^{-1}Q^{(n-1)\eta}\frac{\prod_{1\leq i\leq n}\delta_i^{-1}}{\delta_{\textrm{min}}^{-1}}\nonumber\\
&=Q^{(n+9)\eta-4\nu}  Q^{-1}\prod_{1\leq i\leq n}\delta_i^{-1}.
\end{align}
Since $\bdel\in \cD(t)$ and $Q=2^t$, using \eqref{ziAlt} we obtain 
$$\prod_{1\leq i\leq n}\delta_i^{-1}\leq 2\psi(2^t)^{-1}\leq   2^{1+t(1+\theta)} = 2 Q^{1+\theta}.$$
Choosing $\eta\leq \frac{\nu}{100n}$ and $\theta<\nu$ ensures that
\begin{equation}
\label{eq measure comp 2}
Q^{(n+9)\eta-4\nu}  Q^{-1}\prod_{1\leq i\leq n}\delta_i^{-1}\leq 2Q^{-\nu}= 2^{1-t\nu} = o(1).
\end{equation}
The upshot is that \eqref{Cond 1} is true,
as soon as $t$ is large enough
(which we can assume without loss of generality).
\item Choosing $\eta, \nu$ small enough so that  $10\eta+8\nu\leq \frac{1}{2}$, we have
$${\Delta}/{K}=Q^{10\eta+4\nu-1}\delta_{\textrm{min}}<K\delta_{\textrm{min}}=Q^{-4\nu}.$$
Therefore
$$Q^{-2\eta}\min(K\delta_{\textrm{min}},\Delta/K)=Q^{-2\eta}\Delta/{K}=Q^{8\eta-1}K^{-1}.$$
It follows, for $Q$ large enough depending on $\nu$, that
$$Q^{\eta-1} K^{-1} \leq r=Q^{4\eta-1}K^{-1}\leq  \frac{Q^{-2\eta}}{2}\min(K\delta_{\textrm{min}},\Delta/K).$$
Further, we have ${Q^{-\eta}}K\delta_{\textrm{min}}=Q^{-\eta-4\nu}<1$,
and thus both \eqref{Cond 2} and \eqref{Cond 6} are satisfied.
\item We also have $\eqref{eq Cond 2+}$, since 
$$
r \geq \frac{Q^{\eta-1}}{K}
\geq {Q^{\eta-1}}\delta_{\textrm{min}}.
$$

\item Using \eqref{eq up bd zi} and choosing $\theta<\eta$, we get that 
$$\min_{1\leq i\leq n} \delta_i\geq 2\cdot2^{-t(1+\theta)}\geq 2\cdot Q^{-1-\eta}.$$
Thus \eqref{eq Cond 2++} is true.

\item It is immediate that \eqref{Cond 5} is true.
\end{enumerate}

\subsubsection{The proof}
Recall the discussion in \S\ref{sec final reduc}. To establish Theorem \ref{thm: prob count}, it suffices to show that for every integer $t\geq 1$ and any $\bdel\in \cD(t)$, there exists a set $\cA_\bdel\subset \I$ and a smooth weight $\Omega_{\textrm{g}}=\Omega_{\textrm{g}, \bdel, t}$ which is $\left(\I\setminus \cA_{\bdel}, \delta_{\textrm{min}}/2^t\right)$ adapted such that \eqref{eq: convergence of measures} and \eqref{eq smooth Bdel count} are true, with $\delta_{\textrm{min}}=\delta_1$. 

We now fix $t\geq 1$ and $\bdel\in \cD(t)$ and choose parameters as in \eqref{eq par choice}. Let $W_r$ be as in Lemma \ref{le smooth cut off}.
Let $\left\{(x_i-r, x_i+r)\right\}_{i=1}^{T_r}$ be a collection of intervals satisfying \eqref{eq supp wr} and \eqref{eq sup supp wr}, with $x_i\in U$ for $1\leq i\leq T_r$. We set
$$ \cA_{\bdel}:= \cup _{i=1}^{T_r} (x_i-2r, x_i+2r), \qquad\cB_{\bdel}:=\I\setminus\cA_{\bdel}\,.$$
By \eqref{eq sup supp wr},
$$\lam_1\left(\cA_{\bdel}\right)\leq \lam_1(\fS_{2\Delta, 2K}).$$
Using Theorem \ref{thm quant non div} 
and the estimates \eqref{eq measure comp}, \eqref{eq measure comp 2}, we deduce 
\begin{equation*}
\lam_1\left(\cA_{\bdel}\right)\leq \lam_1(\fS_{2\Delta, 2K})
\ll \left(Q^{(n+9)\eta-4\nu} Q^{-1}\prod_{1\leq i\leq n}\delta_i^{-1}\right)^{\alpha}\ll 2^{-t\nu\alpha},
\end{equation*}
where $\alpha =1/((2l-1)(n+1))$.
Summing up in $\bdel\in \cD(t)$ and then $t\in \mathbb{N}$, and recalling \eqref{eq: size dyadic exponents}, we get
$$\sum_{t\geq 1} \sum_{\bdel\in \cD(t)} \lam_1(\cA_{\bdel})\leq \sum_{t\geq 1} t^{n-1} 2^{-t\nu\alpha}<\infty\,.$$
This establishes \eqref{eq: convergence of measures}. 

Lemma \ref{le smooth cut off} implies 
$\wgood=\Omega_{\textrm{g}}$, as defined in \eqref{def wgood}, is $(\cB_{\bdel}, r)$ adapted. Since 
$$r=\frac{Q^{4\eta-1}}{K}=Q^{4(\eta+\nu)}\frac{\delta_{\textrm{min}}}{Q}\geq \frac{\delta_1}{2^t},$$
the cutoff $\wgood$ is also $(\cB_{\bdel}, \delta_1/2^t)$ adapted. Combining 
Lemmas \ref{le decomp}
and \ref{le error small} yields
$$N_{1, \Omega_{\textrm{g}}}(\bdel, t) \ll_{\bff, \rho} 2^{2t}\bdel^{\times}=2^{2t}\prod_{2\leq i\leq n}\del_i,$$
which establishes \eqref{eq smooth Bdel count} 
and concludes the proof.

\subsubsection*{Funding}
RS was
supported in part by the Deutsche Forschungsgemeinschaft (DFG, German Research Foundation) under Germany's Excellence Strategy-EXC-2047/1-390685813 and an Argelander grant from the University of Bonn.
NT was supported by Deutsche Forschungsgemeinschaft (DFG,
German Research Foundation) - Project-ID 491392403 - TRR 358.
HY was supported by the Leverhulme Trust (ECF-2023-186).

\subsubsection*{Rights}

For the purpose of open access, the authors have applied a Creative Commons Attribution (CC-BY) licence to any Author Accepted Manuscript version arising from this submission.

\appendix

\section{Reduction to Curves via fibring}
\label{sec fibering}

\begin{lemma}
[Fibring lemma for smooth, non-degenerate manifolds]
\label{lem fiber}
Let 
\[
\bfg = (g_1,\dots,g_n): U\to \mathbb{R}^n
\]
be a smooth function in $d$ real variables defined on an open neighbourhood of $\bzero$. Suppose that $\bfg$ is non-degenerate of order $l$ on $U$, and put \mbox{$p = l + 2$.} Let $\eps > 0$ be sufficiently small in terms of $\bg$. Then,
for any $\bh = (h_1, \ldots, h_d) \in \bR^d$ with $\| \bh \|_\infty \le \eps$, 
the curve $ \phi_{\bh}: (-\varepsilon, \varepsilon)\to \mathbb{R}^n$ with coordinates given by
$$
\phi_{\bh,i}(t):= g_i(t^{1+p^d} + h_1, t^{p+p^d} + h_2, \dots, t^{p^{d-1}+p^d} + h_d)
$$
is non-degenerate of order $L$, where $L$ depends on $l$ and $d$.
\end{lemma}

A similar lemma for \textbf{analytic}, non-degenerate manifolds and $\bh = \bzero$ was established in Sprindzhuk's survey \cite[pp.\,9-10]{sprindzhuk1980achievements} and re-proven in a streamlined way in \cite[Appendix C]{Ber2015}, again for analytic, non-degenerate manifolds. Lemma \ref{lem fiber} applies in the more general \textbf{smooth}, non-degenerate setting. We make use of the following technical lemma from \cite{Ber2015}, which will allow us to express multi-dimensional derivatives of $\bfg$ in terms of higher-order, single-variable derivatives of $\phi_{\bh}$.

\begin{lemma}
\label{lem encode}
Let $0<p_0<p$ be positive integers and let $e_p:\Z^d_{\ge0}\to\Z_{\ge0}$ be given by
\begin{equation}
\label{encode map}
e_p(\alpha_1,\dots,\alpha_d):= \sum_{j=1}^d\alpha_j(p^{j-1}+p^d)\,.
\end{equation}
Let
\begin{equation}\label{Sp0}
S_{p_0}:=\{(\alpha_1,\dots,\alpha_d)\in\Z_{\ge0}^d: \alpha_1+\dots+\alpha_d\le p_0\}\,.
\end{equation}
Then
\begin{itemize}
  \item $e_p$ maps $S_{p_0}$ into $\Z_{\ge0}$ injectively, and
  \item $e_p(S_{p_0})\cap e_p(\Z_{\ge0}^d\setminus S_{p_0})=\emptyset$.
\end{itemize}
\end{lemma}
\begin{proof}
See \cite[Lemma 12]{Ber2015}.
\end{proof}

We can now prove the fibring lemma for smooth, non-degenerate manifolds. Our approach is similar to the one in \cite{sprindzhuk1980achievements}, except we no longer have the relatively straightforward equivalence of $\bfg$ being non-degenerate with the component functions $g_1, \ldots, g_n$ and the constant function $1$ being linearly independent. We need to explicitly compute the multi-order derivatives of $\bfg$, match them with appropriate derivatives of $\phi_{\bh}$, and prove they are linearly independent. Another technical obstruction is that we can no longer expand our functions into infinite power series. Instead, we use a multivariable Taylor series expansion up to the appropriate order.

\begin{proof}[Proof of Lemma \ref{lem fiber}]
Since $\bfg$ is non-degenerate of order $l$, there exist multi-indices 
$$\bbet_1 = (\beta_1^{(1)}, \ldots, \beta_d^{(1)}), \bbet_2=(\beta_1^{(2)}, \ldots, \beta_d^{(2)}), \ldots, \bbet_n=(\beta_1^{(n)}, \ldots, \beta_d^{(n)})$$ 
in $\Z_{\geq 0}^d$, with $1\leq \|\bbet_k\|_{1}\leq l$ for all $k$, such that the vectors
\begin{equation*}
    \label{eq span Rn}
    \partial^{\bbet_1}\bfg(\bzero), \ldots, \partial^{\bbet_n}\bfg(\bzero) 
\end{equation*}
span $\mathbb{R}^n$. In the following, we make frequent use of the 
abbreviation
$$
\bx^{\balp}:= x_1^{\alpha_1}\cdots x_d^{\alpha_d}
$$
for $\bx\in \R^d$
and $\balp\in \Z^d_{\geq 0}$.
For consistency, we interpret $0^{0}=1$. 
    
Since $g_1,\dots,g_n$ are smooth, we can use a multivariable Taylor expansion around $\bzero$ on $U$  to express them as 
  $$
g_i(\bfy)=\sum_{\balp\in \mathbb{Z}^d_{\geq 0}: \|\balp\|_1\leq l}\frac{\partial^{\balp} g_i(\bzero)}{\balp!}\, \by^{\balp} \,+ \sum_{\|\balp\|_1=l+1}R_{i, \balp}(\bfy)\, \by^{\balp},
$$
where $R_{i, \balp}$ is a smooth function with bounded derivatives in $U$.

Substituting 
\[
\bfy = \gam(t, \bh) := (t^{1+p^d} + h_1, t^{p+p^d} + h_2, \dots, t^{p^{d-1}+p^d} + h_d)
\]
with $|t| < \eps$, 
we see that
\begin{align*}
\notag
\phi_{\bh,i}(t) &= 
\sum_{
\substack{\balp\in \mathbb{Z}^d_{\geq 0} \\
\|\balp\|_1\leq l}}
\frac{\partial^{\balp} g_i(\bzero)}{\balp!}\, \prod_{j=1}^d(t^{p^{j-1}+p^d} + h_j)^{\alpha_j} \\ \notag
&\qquad + \sum_{\|\balp\|_1=l+1}R_{i, \balp}(\gamma(t, \bh))\,\prod_{j=1}^d(t^{p^{j-1}+p^d} + h_j)^{\alpha_j}\\
&= E(t) + \sum_{
\substack{\balp\in \mathbb{Z}^d_{\geq 0}\\
\|\balp\|_1\leq l}}
\frac{\partial^{\balp} g_i(\bzero)}{\balp!}\, 
t^{e_p(\balp)}\,+  
\sum_{\|\balp\|_1=l+1}
R_{i, \balp}(\gamma(t, \bh)) 
t^{e_p(\balp)},
\label{eq phi taylor exp}
\end{align*}
where $e_p$ is given by \eqref{encode map}. Here
$
E(t) = E(t, \bh) = \sum_{j \le d} h_j E_j(t),
$
where each $E_j(t) = E_j(t, \bh)$ has bounded derivatives.

By Lemma \ref{lem encode}, specifically the injectivity of $e_p$ on $S_{l+1}$ as in \eqref{Sp0} with $p_0=l+1$, we know that the leading powers of $t$ in the  summands above are all distinct. Therefore, given any $\balp\in \mathbb{Z}_{\geq 0}^d$ with $\|\balp\|_1\leq l$, we can differentiate $e_p(\balp)$ many times with respect to $t$ to get
\begin{equation}
\phi_{\bh,i}^{(e_p(\balp))}(t) = \frac{e_p(\balp)!}{\balp!}\partial^{\balp} 
g_i(\bzero)
+ O(\eps).
\end{equation}
In particular, setting $\balp=\bbet_1, \ldots, \bbet_n$, we conclude that
\begin{align*}
&\det \begin{bmatrix}
\phi_{\bh}^{(e_p(\bbet_1))}(t) & \ldots & \phi_{\bh}^{(e_p(\bbet_n))}(t) 
\end{bmatrix}    
\\=& \left(\prod_{1\leq i\leq n} 
\frac{e_p(\bbet_i)!}{\bbet_i!}  \right)\det\, \begin{bmatrix}
\partial^{\bbet_1}\bfg(\bzero) & \ldots & \partial^{\bbet_n}\bfg(\bzero)
\end{bmatrix} + O(\eps).
\end{align*}

As $\eps$ is small, we have
$$\det \begin{bmatrix}
\phi_{\bh}^{(e_p(\bbet_1))}(t) & \ldots & \phi_{\bh}^{(e_p(\bbet_n))}(t) 
\end{bmatrix} \neq 0$$
for all $t\in (-\varepsilon, \varepsilon)$.  
It follows that the curve $\phi_{\bh}$ is non-degenerate of order $\max\left\{e_p(\bbet_1), \ldots, e_p(\bbet_n)\right\}\ll lp^{d}\ll l^{d+1}$ on the interval $(-\varepsilon, \varepsilon)$.
\end{proof}

We can now deduce Theorem \ref{thm: Gallagher conv} from Theorem \ref{thm: main curves}. Let $\cM$ be locally parametrised by a smooth, non-degenerate function 
\[
\bg = (g_1, \ldots, g_n): U \to \bR^n,
\]
for some open ball $U$ in $\bR^d$ centred at the origin.
Let $\psi: \bN \to [0, 1)$ be a monotonic function such that \eqref{eq: series of measures} converges.

Let $\eps > 0$ be sufficiently small in terms of $\bg$. By Lemma~\ref{lem fiber}, if $\bh \in \bR^d$ with $\| \bh \|_\infty \le \eps$, then
\[
\{ \phi_{\bh}(t) : |t| < \eps \}
\]
defines a non-degenerate curve.
Let $r = \eps^{p^{d+1}}$ and $B = B(\bzero, r)$, where $p = l + 2 \ge 3$.
By a change of variables, 
\begin{align*}
&\mu_d(\bg(B) \cap \cW_n^\times(\psi)) \\
&\ll \int_{(-r,r)^{d-1}} \int_{-\eps}^\eps
1_{\cW_n^\times(\psi)}(\phi_{(0,h_2,\ldots,h_s)}(t)) \d t \d h_2 \cdots \d h_s.
\end{align*}
By Theorem \ref{thm: main curves}, the innermost integral vanishes and so
\[
\mu_d(\bg(B) \cap \cW_n^\times(\psi)) = 0.
\]

The same argument applies if we add a constant function to $\bg$, so 
\[
\mu_d(\cW_n^\times(\psi)) = 0.
\]
This completes the proof of Theorem~\ref{thm: Gallagher conv}.

\section{Proof of Lemma \ref{lem red model setting}}
\label{app proof lemma}

We require 
a well-known sub-level version 
of van der Corput's lemma.

\begin{lemma}
\label{le van der Corput}
Let $U\subseteq \R$ be an open set, and let $g:U\rightarrow \R$ be a $C^\ell$ function such that
$
\vert g^{(\ell)}(x)\vert \geq 1
$
for all $x\in U$.
If $z>0$, then
$$
\mu_{1}(\{x\in U: \vert g(x)\vert 
\leq z\})
\leq \fC_\ell z^{1/\ell},
$$
where $\fC_\ell = 2e (\ell+1)!^{1/\ell}$
is a constant depending exclusively on $\ell$.
\end{lemma}

\begin{proof}
This is \cite[Proposition 2.1]{CCW}.
\end{proof}

\begin{proof}
[Proof of Lemma \ref{lem red model setting}]
Fix $\rho >0$ sufficiently small. For $1\leq i\leq n$, let
$$
\cE_{\rho,i} =
\{ x\in \I: 
\vert g_i'(x)\vert < \rho\}.
$$
Define $\I_{\rho} = 
\I\setminus(\cE_{\rho,1} \cup 
\dots \cup \cE_{\rho,n})$.
We shall demonstrate:
\begin{enumerate}[(a)]
\item Each $\cE_{\rho,i}$
is a union
of finitely many open intervals.
\item For $1\leq i\leq n-1$, 
\begin{equation}
\label{meas I eps i}
\lam_1(\cE_{\rho,i})\ll_\bfg \rho^{1/(n-1)}.
\end{equation}
\end{enumerate}

\bigskip
 
Let $m$ be the least absolute determinant of the Wronskian 
\eqref{def wronsk} 
of the curve $\bfg(x)$. By assumption,
\begin{equation}
\label{eq minimal value det}
m= \min_{\dx\in \I}
\vert \det(W_{\bfg}(\dx))\vert >0.
\end{equation}

\begin{enumerate}[(a)]

\item We first suppose that $g_i^{(n)}$ has infinitely many zeros in $\I$. Then it has a limit point of zeros, say $x \in \mathbb I$. By the mean value theorem,
\[
g_i^{(1)}(x)= \dots =
g_i^{(n)}(x)=0.
\]
This contradicts \eqref{eq minimal value det}, which means that this case is impossible. 

Instead, it must be that
$g_i^{(n)}$ has only finitely many zeros in $\I$, say $Z$ many.
Then the mean value theorem yields 
that $g_i^{(n-1)}$ has at most 
$Z+1$ many zeros in $\I$. Applying 
the mean value theorem repeatedly,
we infer that $g'$ has at most 
$Z+n$ many zeros. Since $\rho$ is arbitrarily small, we conclude by compactness that
$\cE_{\rho,i}$ 
is a union of no more than
$Z+n$ many intervals.

\item In view of $(g_i'(\dx),\ldots, g_i^{(n)}(\dx))$
being the $i^\mathrm{th}$ row of $W_{\bfg}(\dx)$,
we infer from \eqref{eq minimal value det}
that for any given point $\dx\in \I$,
there exists $1\leq \kappa_i(\dx)\leq n$
so that 
$ \vert g_i^{(\kappa_i(\dx))}(\dx)\vert
\geq C_\bfg m $
where $C_\bfg>0$ is a constant depending only 
on $\bfg$. By continuity of $g_i^{(\kappa_i(\dx))}$, there exists an open interval $J_{i, x}$ centred at $x$ so that
$\vert g_i^{(\kappa_i(\dx))}(y)\vert
> C_\bfg m/2$ for all $y\in J_{i, x}$.
Since $\{J_{i, x}\}_{x \in \I}$ form an open cover of $\I$, we can extract a finite subcover using compactness. In other words, there exist open intervals $J_{i, 1}, \ldots, J_{i, A}$ such that on each interval,
$$\inf_{y\in J_{i, k}}\vert g_i^{(\kappa_k)}(y)\vert
> \frac{C_\bfg m}{2} \quad \textrm{for some } 1\leq \kappa_k\leq n.$$

We decompose
$$\cE_{\rho, i}=
\bigcup_{k=1}^A (J_{i, k} 
\cap \cE_{\rho, i}).$$
Each term in the union on the right hand side is a union of finitely many open intervals.
Further, since $\rho$ is small, we may assume that
$\rho < C_\bfg m/8$, and so we also have $J_{i, k} \cap \cE_{\rho, i}=\emptyset$ for those $1\leq k\leq A$ for which $\kappa_k=1$. Otherwise, if $y\in J_{i, k}\cap \cE_{\rho, i}$ for such a $k$, we would obtain a contradiction
$$\frac{C_\bfg m}{2}\leq \vert g_i^{(\kappa_k)}(y)\vert=\vert g_i'(y)\vert\leq 2\rho\leq \frac{C_\bfg m}{4}.$$

Thus, without loss of generality, we may assume that $2\leq \kappa_k\leq n$ for each $k$.
We now apply 
Lemma \ref{le van der Corput} to each $J_{i, k}$ with
$$g= ( C_\bfg m/2)^{-1}g_i', \quad
z= ( C_\bfg m/2)^{-1} 2\rho \quad 
\textrm{and} \quad \ell = \kappa_k-1 
\leq n-1
$$ to conclude that
\begin{align*}
\lam_1(\cE_{\rho,i}) \leq \sum_{k=1}^A \lam_1 \left(J_{i, k} \cap \cE_{\rho, i}\right)
\ll ( C_\bfg m)^{-1/(\kap_k - 1)} 
\sum_{k=1}^A\rho^{1/(\kappa_k-1)} 
\ll_\bfg \rho^{\frac{1}{n-1}},
\end{align*}
as required for \eqref{meas I eps i}.
\end{enumerate}
\end{proof}

\section{Proof of Lemma \ref{lem pruning error}}\label{sec proof lemma pruning}
By using Poisson summation in
the $a$ and $q$ variables, we can express 
$E(\bdel, Q)$ as
$$
{\bdel}^{\times}   
\sum_{(c,j_1,\bfj')\in \Z^2 \times \cJ'} 
B(\bfj') \int_{\R^{2}}
e(y \Phi_{\bfj'}( \bfx/y )- y c)
\wgood \left(\frac{\bfx}{y} \right) 
\ome\left(
\frac{y}{Q}\right)
\rd \bfx \, \rd y.
$$
Changing variables twice,
first via $y\mapsto Q y$ and then via 
$\bfx \mapsto Q y \bfx$, 
transforms the integral above into $Q^{2}$ times
\begin{equation}
\label{eq intermediate Error}
  \cI(c, \bfj):=
\int_{\R^{2}}
    y e(Qy [\Phi_\bj(x) - c])
    \wgood(\bfx) 
    \ome(y)
    \,\rd \bfx\, \rd y.  
\end{equation}
Our objective is to truncate the $c$
and $j_1$ variables. To complete the proof, we shall integrate by parts to get rid of the regimes where these parameters are large.

\bigskip

Recall $J_i=Q^\eta/\delta_i$ 
and $J_{\textrm{max}}=\max_i{J_i}=Q^{\eta}\delta_{\textrm{min}}^{-1}$ from \eqref{def J}, and we have assumed that $\mathbf{F}$ is smooth.
Let $L$ be the differential operator as defined in \eqref{def Lop} with 
$$
\psi(\bfx)=\Phi_{\bfj}(x)=\langle (x, \bfF( \bfx )),(j_1,\bfj') \rangle.
$$
Fix a constant $C_{\bfF}>1$, depending only on $\bfF$, 
such that 
\begin{equation}
    \label{eq grad lb}
    \vert \psi' \vert 
    \geq (1/2) \vert j_1 \vert
\end{equation}
holds whenever $\vert j_1 \vert > 
C_\bfF^{1/2} J_{\textrm{max}}$.
Now assume that 
$\vert j_1 \vert > C_\bfF^{1/2} J_{\textrm{max}}$ and let $h=\wgood.$
By Lemma \ref{lem partial}, we have $L^{\circ N} h = 
\sum_{\nu=1}^{V} c_{N,\nu} \wgoodpi,$
where
$$
\wgoodpi = \frac{h^{(\kappa_0)}}{|\psi'|^{\kappa_0}} \prod_{l=1}^M \frac{\psi^{(\kappa_l+1)}}{|\psi'|^{\kappa_l+1}},
$$
with $\kappa_0+\sum_{l=1}^M \kappa_l=N.$ 

Notice that  
$\|h^{(\kap)}\|_\infty \ll_\kap r^{-\kap}$
for each $\kappa\in\mathbb{N}$.
Combining this with the lower bound \eqref{eq grad lb} gives
\begin{equation*}
\left\|\frac{h^{(\kappa_0)}}{(\psi')^{\kappa_0}}\right\|_{\infty}\ll_{\mathbf{F}}
(r\vert j_1 \vert)^{-\kappa_0}.    
\end{equation*}
Observe that $\|\psi^{(\kappa+1)}\|_{\infty}
\ll Q^{\eta}\delta_{\textrm{min}}^{-1}$ for $\kappa\geq 1$,
and thus, using \eqref{eq grad lb} again, we get
$$\left\|\frac{\psi^{(\kappa_l+1)}}{(\psi')^{\kappa_l+1}}\right\|_{\infty}\ll Q^{\eta}\delta_{\textrm{min}}^{-1}\vert j_1 \vert^{-\kappa_l-1}\ll \vert j_1 \vert^{-\kappa_l}.$$
The last inequality follows from our assumption $\vert j_1 \vert >  C_\bfF J_{\textrm{max}}=C_\bfF Q^{\eta}\delta_{\textrm{min}}^{-1}$.
Using \eqref{eq order} and the above observations, 
we infer
$\|h_{N, \nu}\|_\infty\ll r^{-N} \vert j_1 \vert^{-N}.$
Therefore
$$Q^{-N}\|L^{\circ N} \wgood\|_{\infty}\ll_{N} Q^{-N} r^{-N}\vert j_1 \vert^{-N}$$
and
\begin{equation}
\label{eq ibp1}
\cI(c, \bfj)\ll Q^{-N} r^{-N}
     \vert j_1 \vert^{-N}. 
\end{equation}

Next, we estimate the integral (\ref{eq intermediate Error}) for large values of $c$. This time, we integrate by parts in the $y$ variable keeping $\bfx$ fixed. Indeed, the application of Lemma \ref{lem partial}, with 
$$\psi(y)=y 
[ \Phi_{\bfj}(x) - c]),
$$ is much simpler, as all derivatives of order at least 2 of the phase function vanish, and the derivatives of the amplitude $\omega$ are uniformly bounded from above. After possibly enlarging the constant $C_\bfF$, we find that for any $A\in\mathbb{N}$ and any $|c|\geq C_\bfF^{1/2} \max\{J_{\textrm{max}},\vert j_1 \vert\}$,
\begin{equation}
    \label{eq IBP2}
    \cI(c, \bfj)\ll_A (|c|Q)^{-A}.
\end{equation}

Recall that $N\geq 4$, and let $A \in \bN$ be a sufficiently large parameter. 
We have
$$
E(\bdel, Q)-E_{\mathrm{tr}}(\bdel, Q)
\ll {\bdel}^{\times} 
\sum_{\substack{(c, \bfj)\in \Z\times \cJ 
\\ \Vert (c,j_1)\Vert_\infty 
> C_{\bfF} J_{\textrm{max}}}}  Q^{2}
\cI(c, \bfj).
$$
We split the sum above
to express the right-hand side as
${\bdel}^{\times} Q^{2}$ times
\begin{equation}
\sum_{\substack{
(c, \bfj)\in \Z\times \cJ\\ 
\Vert (c,j_1)\Vert_\infty > 
C_{\bfF} J_{\textrm{max}}\\ 
|c|< C_\bfF^{1/2} \vert j_1\vert}}  
\cI(c, \bfj) + 
\sum_{\substack{(c, \bfj)\in \Z\times \cJ 
\\ \Vert (c,j_1)\Vert_\infty > 
C_{\bfF} J_{\textrm{max}}\\ |c|\geq  
C_\bfF^{1/2} \vert j_1 \vert}} \cI(c, \bfj). \label{eq align t}    
\end{equation}
Using \eqref{eq ibp1} and taking $C_{\bfF}$ sufficiently large, the first term in \eqref{eq align t} can be bounded from above by $C_{\bfF}$ times
\begin{align*}
 \sum_{\substack{\bfj\in  
 \cJ \\ 
 \vert j_1 \vert > C_{\bfF}^{1/2} J_{\textrm{max}} }}  
 \vert j_1 \vert \left(Qr \vert j_1 \vert\right)^{-N}
   & \ll 
   J_{\textrm{max}}^{n-1} 
   Q^{-N}r^{-N} J_{\textrm{max}}^{2-N} =
   J_{\textrm{max}}^{n+1}  (QrJ_\max)^{-N}.
\end{align*}
Similarly, we can use \eqref{eq IBP2} to estimate the second term in \eqref{eq align t}, say with
$A = 2(n+1+N)$, as follows:
\begin{align*}
\sum_{\substack{(c,\bfj)\in 
\Z\times \cJ \\ 
\Vert (c,j_1)\Vert_\infty 
> C_{\bfF} J_{\textrm{max}}\\ 
|c|\geq  C_\bfF^{1/2} \vert j_1\vert}}  
(|c| Q)^{-A}
   & \ll_A 
   J_{\textrm{max}}^{n-1} 
   \sum_{|c| > 
   C_{\bfF} 
   J_{\textrm{max}}} |c| (|c| Q)^{-A} \\
   & \ll_A 
   J_{\textrm{max}}^{n-1} Q^{-A} 
   J_{\textrm{max}}^{2-A}
   \ll_N Q^{-2-2N}.
\end{align*}
Thus,
$$E(\bdel, Q) - 
E_{\mathrm{tr}}(\bdel, Q)
    \ll_N  {\bdel}^{\times} J_{\textrm{max}}^{n+1} Q^{2} (QrJ_{\textrm{max}})^{-N} 
+{\bdel}^{\times}Q^{-2N}.$$

\end{document}